\documentclass[brochure,12pt]{bourbaki}
\usepackage[matrix,arrow]{xy}
\usepackage{amssymb,amsfonts,amsmath,footnote}
\usepackage[francais]{babel}
\usepackage{enumerate,mathrsfs,amsthm}
\usepackage[utf8]{inputenc}
\usepackage[T1]{fontenc}                                               
\newtheorem*{thmm}{Th\'eor\`eme}
\newtheorem{thm}{Th\'eor\`eme}[subsection]
\newtheorem{cor}[thm]{Corolaire} 
\newtheorem{lem}[thm]{Lemme}
\newtheorem{propp}[thm]{Proposition}
\newtheorem{fact}[thm]{Fait}
\theoremstyle{definition}
\newtheorem{defn}[thm]{D\'efinition}
\newtheorem{exe}[thm]{Exemple}

\numberwithin{equation}{section}

\newcommand{\N}{\mathbf{N}}
\newcommand{\Z}{\mathbf{Z}}
\newcommand{\R}{\mathbf{R}}
\newcommand{\GL}{\textnormal{GL}}
\newcommand{\TM}{\textnormal{TM}}
\newcommand{\Ker}{\textnormal{Ker}}

\newcommand{\Homeo}{\textnormal{Homeo}}
\newcommand{\eps}{\varepsilon}
\newcommand{\lb}{[\![}
\newcommand{\rb}{]\!]}
\newcommand{\os}{_\textnormal{s}}
\newcommand{\Clo}{\textnormal{Clo}}
\newcommand{\vpi}{\varphi}
\newcommand{\Alt}{\mathfrak{A}}
\newcommand{\Supp}{\textnormal{Supp}}
\newcommand{\rt}{\rightthreetimes}
\newcommand{\mo}{\textnormal{mod}}
\newcommand{\Hom}{\textnormal{Hom}}
\newcommand{\Cyl}{\textnormal{Cyl}}

\date{Janvier 2013}
\bbkannee{65\`eme ann\'ee, 2012-2013}
\bbknumero{1064}
\title{Groupes pleins-topologiques}
\subtitle{d'apr\`es Matui, Juschenko, Monod, ...}
\author{Yves de CORNULIER}
\address{CNRS -- Universit\'e Paris-Sud 11\\
D\'epartement de Math\'ematiques\\
UMR 8628 du CNRS\\
B\^atiment 425\\
91405 Orsay Cedex}
\email{yves.cornulier@math.u-psud.fr}

\begin{document}
\maketitle

\section{Introduction}

Le but de cet expos\'e est de d\'ecrire la construction r\'ecente et \'el\'egante de groupes infinis de type fini, simples et moyennables. 
Elle d\'ecoule du th\'eor\`eme ci-dessous. On appelle {\em espace de Cantor} un espace hom\'eo\-morphe \`a l'ensemble triadique de Cantor. 
Soit $\vpi$ un auto\-hom\'eo\-mor\-phisme d'un espace de Cantor $X$, c'est-\`a-dire un hom\'eo\-mor\-phisme de $X$ sur lui-m\^eme. On suppose $\vpi$ minimal, au sens o\`u toutes ses orbites sont denses. Par d\'efinition, son groupe {\em plein-topologique} $\lb\vpi\rb$, 
est le groupe (discret!) des auto\-hom\'eo\-mor\-phismes de $X$ qui co\"\i ncident localement avec une puissance de $\vpi$ (voir au \S\ref{defpl} une d\'efinition plus g\'en\'erale et plus d\'etaill\'ee).

\begin{thmm}
Le groupe plein-topologique $\lb\vpi\rb$ est infini, ainsi que son sous-groupe d\'eriv\'e $\lb\vpi\rb'$; de plus
\begin{itemize}
\item {\rm  (Matui, 2006)} le groupe d\'eriv\'e $\lb\vpi\rb'$ est un groupe simple;
\item {\rm (Matui, 2006)} si $\vpi$ est un sous-d\'ecalage topologique (voir \S\ref{sdo}), alors le groupe d\'eriv\'e $\lb\vpi\rb'$ est un groupe de type fini;
\item {\rm (Juschenko--Monod, 2012; conjectur\'e par Grigorchuk--Medynets \cite{GrMe})} le groupe $\lb\vpi\rb$ est moyennable (et donc $\lb\vpi\rb'$ aussi).
\end{itemize}
\end{thmm}

Le groupe plein-topologique $\lb\vpi\rb$ est un cas particulier d'une 
construction de W. Krieger (1980) \cite[p.~88]{Kri} mais n'y est pas 
\'etudi\'e dans le cas d'une action d'un auto\-hom\'eo\-mor\-phisme 
minimal. Il est ensuite\footnote{{\it (Ajout, f\'evrier 2020)} \v St\v 
ep\' anek et Rubin (1989) \cite{SR} ont \'egalement introduit le groupe 
plein-topologique d'un sous-groupe d'autohom\'eomorphismes d'un espace 
compact totalement s\'epar\'e: voir l'appendice~\ref{ajout}.} 
implicitement introduit, dissimul\'e sous le nom {\og$\Gamma$\fg}, en 
termes de $C^\ast$-alg\`ebres par Putnam (1989) \cite[Sec.~5]{Pu}, en 
\'etudiant dans un cadre topologique des objets introduits par Vershik 
\cite{Ve1,Ve2} dans un contexte mesurable. Le groupe plein-topologique 
d'un auto\-hom\'eo\-mor\-phisme minimal est ensuite \'etudi\'e de 
mani\`ere syst\'ematique par E.~Glasner et B.~Weiss \cite{GW} (sous le 
nom {\og groupe plein-fini\fg}), Skau \cite{Sk} et surtout Giordano, 
Putnam et Skau (1999) dans \cite{GPS} (et d\'ej\`a, de mani\`ere 
implicite mais profonde, dans l'article important \cite{GPS95}). Ces 
derniers montrent notamment, en utilisant un r\'esultat de Boyle 
\cite{Bo}, que la classe d'isomorphie du groupe $\lb\vpi\rb$ 
caract\'erise la paire $\{\vpi,\vpi^{-1}\}$ \`a conjugaison pr\`es dans 
le groupe des auto\-hom\'eo\-mor\-phismes de l'espace de Cantor. Ceci a 
\'et\'e \'etendu \`a $\lb\vpi\rb'$ par Bezuglyi et Medynets \cite{BeM}, 
avec une m\'ethode plus directe. De l'existence d'un continuum 
d'auto\-hom\'eo\-mor\-phismes minimaux de l'espace de Cantor deux \`a 
deux non conjugu\'es et qui sont des sous-d\'ecalages (par exemple, 
associ\'es aux rotations irrationnelles, voir l'exemple plus bas), il 
d\'ecoule que les groupes simples de type fini obtenus $\lb\vpi\rb'$ 
forment \'egalement un continuum de classes d'isomorphie.

\subsection*{Moyennabilit\'e}
Un groupe $\Gamma$ est {\em moyennable} s'il poss\`ede une moyenne (c'est-\`a-dire une probabilit\'e finiment additive d\'efinie sur toutes les parties) invariante par translations \`a gauche.
Un des int\'er\^ets de la moyennabilit\'e est qu'elle peut aussi bien \^etre caract\'eris\'ee de mani\`ere analytique, probabiliste, dynamique et g\'eom\'etrique. \`a titre d'illustration, en voici diverses caract\'erisations \'equivalentes (qu'il n'est pas n\'ecessaire de d\'echiffrer pour comprendre la suite; voir notamment \cite{Gre,Eym,Pat}). 
\begin{itemize}
\item $\Gamma$ admet des parties presque invariantes: pour toute partie finie $S\subset\Gamma$ et $\eps>0$ il existe une partie finie non vide $F$ de $\Gamma$ telle que $\#(SF\smallsetminus F)/\#F\le\eps$.
\item Toute action de $\Gamma$ par auto\-hom\'eo\-mor\-phismes sur un espace compact non vide pr\'eserve une mesure de probabilit\'e invariante sur les bor\'eliens.
\item Toute action affine continue de $\Gamma$ sur un convexe compact non vide d'un espace localement convexe admet un point fixe.
\item $\Gamma$ n'admet pas de d\'ecomposition paradoxale (voir remarque \ref{parado}).
\item La repr\'esentation r\'eguli\`ere de $\Gamma$ sur $\ell^2(\Gamma)$ admet des vecteurs unitaires presque invariants. 
\item L'alg\`ebre de Banach $\ell^1(\Gamma)$ est nucl\'eaire.
\item L'alg\`ebre de von Neumann de $\Gamma$ est hyperfinie.
\item Le morphisme canonique de la $C^\ast$-alg\`ebre maximale de $\Gamma$ vers sa $C^\ast$-alg\`ebre r\'eduite est un isomorphisme.
\item Pour toute mesure de probabilit\'e sym\'etrique \`a support fini $S\subset \Gamma$ contenant 1, la marche al\'eatoire donn\'ee au temps $n$ par multiplication \`a droite par un \'el\'ement choisi au hasard dans $S$, a une probabilit\'e de retour en 1 d\'ecroissant sous-expo\-nen\-tiellement vers 0.
\end{itemize}

Les groupes moyennables constituent une classe stable par passage aux sous-groupes, quotients, extensions et limites directes. La plus petite classe de groupe stable par ces op\'erations contenant les groupes finis et ab\'eliens est appel\'ee classe des groupes {\em \'el\'ementairement moyennables}. De fait, ce sont les exemples {\og imm\'ediats\fg}  de groupes moyennables, comprenant notamment les groupes virtuellement r\'eso\-lubles (c'est-\`a-dire, ayant un sous-groupe d'indice fini r\'esoluble). Cependant, il est facile de v\'erifier que tout groupe de type fini \'el\'ementairement moyennable et non trivial admet un quotient fini non trivial (et, s'il est infini, admet un quotient infini virtuellement ab\'elien non trivial), si bien que ces groupes sont tr\`es loin d'\^etre simples. S'il est connu depuis le d\'ebut de l'\'etude de la croissance des groupes que tout groupe de croissance sous-exponentielle est moyennable, les premiers exemples non triviaux (\`a croissance dite interm\'ediaire) ont \'et\'e obtenus par Grigorchuk en 1984 \cite{Gri}. Ce sont les premiers exemples connus de groupes moyennables et non \'el\'ementairement moyennables. Ce sont des groupes d'automorphismes d'arbres enracin\'es de degr\'e fini et \`a ce titre ils sont r\'esiduellement finis (et donc loin d'\^etre simples). La plus petite classe de groupes stable par les op\'erations d\'ecrites pr\'ec\'edemment contenant les groupes \`a croissance sous-exponentielle est appel\'ee classe des groupes {\em sous-exponentiellement moyennables}; elle contient les groupes \'el\'ementairement moyennables mais est strictement plus grosse puisqu'elle contient \'egale\-ment les groupes \`a croissance interm\'ediaire. Les premiers exemples de groupes moyennables mais pas sous-exponentiellement moyennables, comprenant le {\og groupe de la basilique\fg}, ont \'et\'e introduits par Grigorchuk et Zuk \cite{GZ}; leur moyennabilit\'e a \'et\'e d\'emontr\'ee par Bartholdi et Virag \cite{BV}, en utilisant les marches al\'eatoires; ils sont \'egale\-ment r\'esiduellement finis. Leur m\'ethode a \'et\'e largement g\'en\'eralis\'ee et \'etendue \cite{Kai,Ers,Bri,BKN}, par des m\'ethodes \'egalement bas\'ees sur les marches al\'ea\-toires et s'appliquant \`a des groupes agissant naturellement sur des arbres enracin\'es de valence finie.

On ne connaissait pas, pr\'ec\'edemment, de groupe de type fini moyennable non trivial sans quotient fini non trivial, et donc, a fortiori, aucun qui soit infini et simple. Les exemples $\lb\vpi\rb'$ sont \`a croissance exponentielle, par un r\'esultat de Matui \cite{Ma2}, et il en d\'ecoule imm\'ediatement qu'ils ne sont pas sous-exponentiellement moyennables (voir le fait \ref{factsim}).

Le crit\`ere de moyennabilit\'e utilis\'e par Juschenko et Monod \cite{JM} est le suivant. Rappelons qu'une permutation $\sigma$ de $\Z$ est \`a {\em d\'eplacement born\'e} si $|\sigma(n)-n|$ est born\'e ind\'ependamment de $n\in\Z$, et que deux parties $A,B\subset\Z$ sont {\em commensurables} si la diff\'erence sym\'etrique $A\bigtriangleup B$ est finie. 

\begin{thmm}[Crit\`ere de Juschenko-Monod]
Soit $\Gamma$ un groupe agissant sur $\Z$ par permutations \`a d\'eplacement born\'e. On suppose que le stabilisateur de $\N$ est moyennable. 
Alors $\Gamma$ est moyennable.
\end{thmm}

Dans le cas d'un auto\-hom\'eo\-mor\-phisme minimal d'un espace de Cantor, le th\'eor\`eme pr\'ec\'edent s'applique \`a $\lb\vpi\rb$, qui agit sur $\Z$ via l'identification de $\Z$ \`a une orbite d'un \'el\'ement quelconque $x_0$ dans l'espace de Cantor: 
$\psi\cdot n$, pour $\psi\in\lb\vpi\rb$ et $n\in\Z$, est l'unique entier $m$ tel que $\psi(\vpi^n(x_0))=\vpi^m(x_0)$. Putnam ayant montr\'e \cite{Pu} que le stabilisateur dans $\lb\vpi\rb$ de toute orbite positive $\{\vpi^n(x):n\in\N\}$ est localement fini (au sens o\`u toute partie finie engendre un sous-groupe fini), le th\'eor\`eme implique que $\lb\vpi\rb$ est moyennable.

\subsection*{Exemples}
Un exemple explicite de sous-d\'ecalage minimal $\vpi$ s'obtient comme ceci: on 
consid\`ere un irrationnel $\alpha$ et on consid\`ere, sur le cercle $\R/\Z$, la rotation $r_\alpha:z\mapsto z+\alpha$. On peut d\'efinir un espace $[\R/\Z]_{\Z\alpha}$, hom\'eomorphe \`a l'espace de Cantor, muni d'une projection sur le cercle par laquelle la rotation $r_\alpha$ se rel\`eve canoniquement en un auto\-hom\'eo\-mor\-phisme minimal $\vpi=\tilde{r}_\alpha(x)$. Une mani\`ere de l'obtenir est de remplacer chaque point $x$ de l'orbite $\Z\alpha\subset\R/\Z$ par une paire $\{x^-,x^+\}$, et de munir l'espace ainsi obtenu de l'ordre cyclique \'evident et de la topologie associ\'ee. Ainsi si $x\notin\Z\alpha$, on d\'efinit $\tilde{r}_\alpha(x)=x+\alpha$ et si $x\in\Z\alpha$ on d\'efinit $\tilde{r}_\alpha(x^\pm)=(x+\alpha)^\pm$; l'espace $[\R/\Z]_{\Z\alpha}$ est m\'etrisable, compact, non vide, totalement discontinu et sans point isol\'e et donc est hom\'eomorphe \`a l'espace de Cantor; l'auto\-hom\'eo\-mor\-phisme $\tilde{r}_\alpha$ est minimal. On v\'erifie alors que pour tout $\eps\in\Z\alpha\,\cap\mathopen]0,1/2\mathclose[\subset\R/\Z$, les translat\'es de l'ouvert ferm\'e $[0^+,\eps^-]$ s\'eparent les points de $[\R/\Z]_{\Z\alpha}$, si bien que $([\R/\Z]_{\Z\alpha},\varphi_\alpha)$ est un sous-d\'ecalage (voir le fait \ref{fait_sd}). Cette construction est due \`a Keane \cite[\S 5]{Kea} et le groupe plein-topologique associ\'e est consid\'er\'e sp\'ecifiquement par Matui dans \cite[Example 6.2]{Ma1}. Notons que $r_\alpha$ est uniquement ergodique (th\'eor\`eme d'\'equidistribution de Kronecker-Weyl, voir \cite[Proposition 4.2.1]{KH}) et les valeurs propres de l'op\'erateur unitaire associ\'e sur $L^2([\R/\Z]_{\Z\alpha})=L^2(\R/\Z)$ sont les $e^{2i\pi n\alpha}$ pour $n\in\Z$; en particulier les groupes cycliques $\langle r_\alpha\rangle\subset \Homeo([\R/\Z]_{\Z\alpha})$ pour $\alpha$ irrationnel dans $[0,1/2]$ sont deux \`a deux non conjugu\'es et donnent donc, par le r\'esultat de Boyle et Giordano-Putnam-Skau, des groupes pleins-topologiques deux \`a deux non isomorphes.   

\medskip

L'exemple pr\'ec\'edent est sous-groupe du groupe $\textnormal{IET}$ des \'echanges d'intervalles, introduit par Keane \cite{Kea}, habituellement d\'efini comme l'ensemble des translations par morceaux continues \`a droite de $\R/\Z$ mais qui, comme Keane l'a observ\'e \cite[\S 5]{Kea}, peut s'interpr\'eter comme groupe d'auto\-hom\'eo\-mor\-phismes d'un compact totalement discontinu. La question de la moyennabilit\'e du groupe entier $\textnormal{IET}$ est ouverte. La r\'esoudre consisterait \`a d\'emontrer un r\'esultat de moyennabilit\'e pour des groupes pleins-topologiques associ\'es \`a certaines actions de groupes ab\'eliens de type fini. Cependant, le th\'eor\`eme de Juschenko-Monod ne se g\'en\'eralise pas sans hypoth\`ese suppl\'ementaire: une construction \'el\'ementaire dans \cite{EM} montre que le groupe plein-topologique associ\'e \`a une action libre et minimale de $\Z^2$ sur un espace de Cantor peut contenir un groupe libre non ab\'elien.

\subsection*{Plan} (Ci-dessous, $\vpi$ est un auto\-hom\'eo\-mor\-phisme minimal de l'espace de Cantor.)
 
On va commencer par introduire les notions g\'en\'erales en d\'etail, dans la partie \ref{sdgp}; et \'enoncer quelques faits g\'en\'eraux (non sp\'ecifiques aux actions de $\Z$), notamment le fait \ref{fait_sn} qui montre que les groupes pleins-topologiques contiennent {\og beaucoup\fg}  d'\'el\'ements de torsion et montre en particulier que $\lb\vpi\rb'$ est un groupe infini. 

La partie \ref{pa} contient les principaux r\'esultats alg\'ebriques connus sur le groupe $\lb\vpi\rb'$. En particulier, les r\'esultats de Matui mentionn\'es plus haut sont \'etablis au \S\ref{simplicite} pour la simplicit\'e et au \S\ref{type} pour la type-finitude dans le cas d'un sous-d\'ecalage minimal.

D'autres r\'esultats sont prouv\'es au \S\ref{approxi}, notamment le r\'esultat, \'egalement d\^u \`a Matui, que $\lb\vpi\rb'$ n'est jamais de pr\'esentation finie, ainsi que le fait, d\^u \`a Grigorchuk et Medynets, qu'il n'est pas finiment approximable (d\'efinition \ref{lef}). Au \S\ref{sgl}, on prouve le r\'esultat, aussi d\^u \`a Matui, que pour un sous-d\'ecalage minimal infini, $\lb\vpi\rb'$ contient un semi-groupe libre \`a 2 g\'en\'erateurs et est donc \`a croissance exponentielle.

La partie \ref{smoy} contient le r\'esultat de moyennabilit\'e de Juschenko-Monod. Elle est ind\'ependante de la pr\'ec\'edente en ce qui concerne la preuve du crit\`ere. Le fait d'appliquer le crit\`ere est lui bas\'e sur le lemme de Putnam, qui se trouve au \S\ref{abelianise}.

\`a titre d'illustration de l'importance de l'hypoth\`ese de minimalit\'e dans certains des r\'esultats ci-dessus, quelques exemples de groupes pleins-topologiques $\lb\vpi\rb$ pour $\vpi$ sous-d\'ecalage non minimal sont \'etudi\'es au~\S\ref{enm}.

\subsection*{Probl\`emes et questions ouvertes}

On consid\`ere le groupe plein-topologique associ\'e \`a un sous-d\'ecalage minimal infini sur $\Z$. Rappelons que son groupe d\'eriv\'e $\lb\vpi\rb'$ est, par un r\'esultat de Matui, un groupe simple de type fini.

\renewcommand{\labelenumii}{(\arabic{enumi}\alph{enumii})}
\renewcommand{\labelenumii}{(\theenumi\theenumii)}

\begin{enumerate}[(1)]
\item Estimer la fonction de F\o lner de $\lb\vpi\rb'$, peut-\^etre en discutant sur $\vpi$. Rappelons qu'apr\`es choix d'un syst\`eme g\'en\'erateur fini $S$ sym\'etrique et contenant 1, elle associe \`a $n\ge 1$ le nombre \[\textnormal{F\o l}(n)=\inf\{\# K\mid \#(SK\smallsetminus K)/\#(K)\le 1/n\},\]$K$ parcourant les parties finies non vides du groupe. De m\^eme, estimer la fonction de F\o lner dans les boules \[\textnormal{BF\o l}(n)=\inf\{\textnormal{diam}(K)\mid \#(SK\smallsetminus K)/\#(K)\le 1/n\}\] et la probabilit\'e de retour des marches al\'eatoires sur le graphe de Cayley de $\lb\vpi\rb'$.
\item Donner des r\'esultats g\'en\'eraux sur la structure des sous-groupes de $\lb\vpi\rb$. Par exemple:
\begin{enumerate}
\item \'etant moyennable, $\lb\vpi\rb$ ne contient pas de sous-groupe libre non ab\'elien. Donner une preuve directe de ce fait.
\item\label{sgtf} Est-ce que tous les sous-groupes de pr\'esentation finie (resp.\ r\'esi\-duel\-lement finis, resp.\ sans torsion) de $\lb\vpi\rb'$ sont \'el\'e\-men\-tai\-rement moyennables?
\item\label{expol} Est-il vrai que tout sous-groupe de type fini de $\lb\vpi\rb$ est \`a croissance exponentielle ou polynomiale? que tout sous-groupe infini de type fini poss\`ede un \'el\'ement d'ordre infini?  
\item\label{alter} Est-il vrai que tout sous-groupe de $\lb\vpi\rb$ soit contient un sous-groupe isomorphe \`a $\lb\psi\rb$ (pour un autre sous-d\'ecalage minimal $\psi$), soit est \'el\'ementairement moyennable? (Cela impliquerait une r\'eponse positive \`a toutes les questions de (\ref{sgtf}) et (\ref{expol}).)
\item Peut-on donner des r\'esultats g\'en\'eraux sur la structure des sous-groupes \'el\'ementairement moyennables de $\lb\vpi\rb$?
\item Classifier les morphismes entre les diff\'erents groupes $\lb\vpi\rb$ ou $\lb\vpi\rb'$ (ici $\vpi$ peut \^etre plus g\'en\'eralement un auto\-hom\'eo\-mor\-phisme minimal d'un espace de Cantor). En particulier, le groupe $\lb\vpi\rb'$ a-t-il des endomorphismes injectifs non surjectifs (autrement dit, en termes savants, est-il non cohopfien)? [Le cas des isomorphismes et automorphismes est bien compris, principalement par les r\'esultats de Giordano, Putnam et Skau \cite{GPS}, voir \S\ref{isoa}.]
\end{enumerate}
\item Peut-on classifier les groupes $\lb\vpi\rb'$ \`a quasi-isom\'etrie pr\`es? Peut-on caract\'eriser les groupes (de type fini, ou plus g\'en\'eralement localement compacts et compactement engendr\'es) qui leur sont quasi-isom\'etriques?
\item  Est-ce que $\lb\vpi\rb'$ admet une action propre sur un complexe cubique CAT(0) (sans hypoth\`ese de dimension finie)? Peut-on d\'eterminer les sous-groupes de codimension~1 dans les groupes $\lb\vpi\rb'$? Rappelons qu'un sous-groupe $H$ est dit de codimension~1 si le graphe de Schreier $\lb\vpi\rb'/H$ a au moins 2 bouts, voir par exemple Sageev \cite{Sag}. Les stabilisateurs dans $\lb\vpi\rb'$ des points de $X$ sont de codimension 1 puisque le graphe de Schreier est quasi-isom\'etrique \`a $\Z$.
\item Est-ce que le groupe $\lb\vpi\rb$ a une complexit\'e de d\'ecomposition finie? (Voir \cite{GTY}.)
\item Le groupe $\lb\vpi\rb'$ contient-il un arbre r\'egulier trivalent quasi-isom\'etriquement plong\'e? un semi-groupe libre quasi-isom\'etriquement plong\'e?
\item Le groupe $\lb\vpi\rb'$ a-t-il une croissance exponentielle uniforme? A-t-il un diam\`etre semi-libre (resp.\ sans torsion) uniforme (c'est-\`a-dire, existe-t-il un entier $N$ tel que pour tout syst\`eme g\'en\'erateur, la $N$-boule contient un couple d'\'el\'ements engendrant librement un semi-groupe libre (resp.\ contient un \'el\'ement sans torsion)?
\item Soit $s_\vpi$ cardinal minimal d'une famille g\'en\'eratrice de $\lb\vpi\rb'$. Existe-t-il des sous-d\'ecalages minimaux avec $s_\vpi$ arbitrairement grand?
\end{enumerate}

\noindent\textbf{Remerciements.}  Je remercie Mikael de la Salle pour une simplification substantielle de la preuve de la proposition \ref{mainprop2} et Pierre Py pour une correction dans la preuve du th\'eor\`eme \ref{dersim}. Je suis reconnaissant envers toutes les personnes m'ayant signal\'e diverses coquilles, omissions ou impr\'ecisions, et notamment Laurent Bartholdi, Pierre de la Harpe et Ralph Strebel pour leur relecture critique attentive. Je remercie Hiroki Matui et Jean-Fran\c cois Quint pour des discussions utiles.

 \setcounter{tocdepth}{1}

\section{Sous-d\'ecalages et groupes pleins-topologiques}\label{sdgp}

\subsection{Sous-d\'ecalages et odom\`etres}\label{sdo}
Si $X$ est un espace topologique, on note $\Homeo(X)$ le groupe des hom\'eo\-mor\-phismes de $X$ sur lui-m\^eme, ou {\em auto\-hom\'eo\-mor\-phismes} de $X$. On consid\`ere un espace topologique $X$ muni d'une action par auto\-hom\'eo\-mor\-phismes d'un groupe discret $\Gamma$; on parlera souvent du couple $(\Gamma,X)$.

On supposera souvent que $X$ est {\em totalement s\'epar\'e}, au sens o\`u les fonctions continues de $X$ vers les espaces discrets s\'eparent les points de $X$, ou de mani\`ere \'equivalente, les clouverts de $X$ forment une base de la topologie; c'est un peu plus fort que totalement discontinu (voir \cite[Ex.~129]{SS} pour un contre-exemple sous-ensemble du plan {\og tipi de Cantor \'epoint\'e\fg}), et \'equivalent pour des espaces localement compacts \cite[II.\S 4, cor.\ de la prop.~6]{Bou}. 

Commen\c cons par introduire des exemples fondamentaux.

\begin{defn}
Soit $d\in\N^*$, soient $A_d$ un alphabet \`a $d$ lettres et $\Gamma$ un groupe discret quelconque. L'espace $A_d^\Gamma$, muni de la topologie produit, est compact totalement s\'epar\'e (si $d\ge 2$ et $\Gamma$ est infini d\'enombrable, il est hom\'eomorphe \`a un espace de Cantor). Il admet une action naturelle de $\Gamma$ par d\'ecalage: si $f$ est une fonction de $\Gamma$ dans $A_d$ et $g\in\Gamma$, on d\'efinit $g\cdot f$ par $g\cdot f(\gamma)=f(g^{-1}\gamma)$. Muni de l'action de $\Gamma$, l'espace $A_d^\Gamma$ est appel\'e d\'ecalage (\`a $d$ lettres) sur $\Gamma$. Un ferm\'e $\Gamma$-invariant est appel\'e {\em sous-d\'ecalage topologique} \`a $d$ lettres sur $\Gamma$ (on omettra syst\'ematiquement l'adjectif topologique). Par extension et abus de langage, un espace $X$ muni d'une action de $\Gamma$ par auto\-hom\'eo\-mor\-phismes sera appel\'e {\em sous-d\'ecalage} s'il est hom\'eo\-morphe, de mani\`ere $\Gamma$-\'equivariante, \`a un sous-d\'ecalage au sens ci-dessus.
Lorsqu'on parlera de d\'ecalage ou sous-d\'ecalage, on supposera toujours implicitement le nombre de lettres fini. 
\end{defn}

Le fait suivant donne une caract\'erisation \'el\'ementaire et tr\`es utile des sous-d\'ecalages. On rappelle qu'un {\em clouvert} est par d\'efinition un ouvert ferm\'e. On note $\Clo(X)$ l'ensemble des clouverts de $X$, qui est une sous-alg\`ebre bool\'eenne de~$2^X$.

\begin{fact}\label{fait_sd}
Un espace $X$ muni d'une action de $\Gamma$ par auto\-hom\'eo\-mor\-phismes est un sous-d\'ecalage si et seulement si $X$ est compact et il existe une partition finie de $X$ en clouverts, dont les $\Gamma$-translat\'es s\'eparent les points. Plus pr\'ecis\'ement, il est hom\'eo\-morphe, de mani\`ere $\Gamma$-\'equivariante, \`a un sous-d\'ecalage \`a $d$ lettres si et seulement s'il existe une partition de $X$ en (au plus) $d$ clouverts, dont les $\Gamma$-translat\'es s\'eparent les points.
\end{fact}

\begin{rema}La condition du fait \ref{fait_sd} sur la partition se reformule en disant que l'alg\`ebre bool\'eenne $\Clo(X)$ est de type fini comme alg\`ebre munie d'une action de $\Gamma$; la condition plus pr\'ecise dit qu'elle est engendr\'ee par $d-1$ \'el\'ements disjoints.\end{rema}

\begin{proof}[D\'emonstration du fait \ref{fait_sd}]
Dans un sens, on remarque que, dans $A_d^\Gamma$, les translat\'es des clouverts $\Omega_i=\{w\in A_d^\Gamma\mid w(1)=i\}$, pour $i\in A_d$, s\'eparent les points. R\'eciproquement, si $X=\bigsqcup_{i=1}^dX_i$ avec $X_i$ clouvert, on d\'efinit, pour $x\in X$, la lettre $\iota(x)$ comme l'unique $i\in A_d$ tel que $x\in X_i$; alors
on remarque que l'application continue $\Gamma$-\'equivariante
\begin{eqnarray*}
\Psi:X & \to  & A_d^\Gamma\\
x & \mapsto & (\iota(\gamma^{-1} x))_{\gamma\in\Gamma}
\end{eqnarray*}
est injective si et seulement si les translat\'es de la partition $(X_i)$ s\'eparent les points. La compacit\'e de $X$ implique que $\Psi$ est un hom\'eo\-mor\-phisme sur son image, qui est un ferm\'e $\Gamma$-invariant de $A_d^\Gamma$.
\end{proof}

On a un autre type d'exemple, en un certain sens diam\'etralement oppos\'e.

\begin{defn}\label{dodo}On dit que $(\Gamma,X)$ est un {\em odom\`etre} si $X$ est compact totalement s\'epar\'e, et si l'action de $\Gamma$ sur l'ensemble $\Clo(X)$ des clouverts de $X$ a toutes ses orbites finies.
\end{defn}

Un exemple prototypique est l'action par addition de $\Z$ sur l'anneau des nombres entiers $p$-adiques $\Z_p$; il est souvent d\'ecrit de mani\`ere plus explicite, l'action du g\'en\'erateur positif de $\Z$ consistant \`a effectuer effectivement l'addition de 1 \`a un entier $p$-adique, en posant les retenues!


\newcounter{foo}
\begin{fact}\label{carodo}Soit $X$ un espace topologique muni d'une action par auto\-hom\'eo\-mor\-phismes de $\Gamma$. \'equivalences:
\begin{enumerate}[~(i)]
\item\label{i_odo} $X$ est un odom\`etre;
\item\label{i_shift} $X$ est compact totalement s\'epar\'e et toute application \'equivariante continue de $X$ vers un sous-d\'ecalage a une image finie;
\item\label{i_proj} $X$ est une limite projective filtrante de $\Gamma$-ensembles finis.
\end{enumerate}
Si de plus $X$ admet une $\Gamma$-orbite dense, (\ref{i_proj}) se reformule en:
\begin{enumerate}[(i')]
\addtocounter{enumi}{2}
\item\label{e_quo} $X$ est isomorphe, comme espace topologique muni d'une action de $\Gamma$ (et canoniquement, apr\`es choix d'un point de $X$ dont l'orbite est dense) au quotient \`a droite $\hat{\Gamma}/K$ du compl\'et\'e profini de $\Gamma$ par un sous-groupe ferm\'e $K$.
\end{enumerate}
Si en outre $\Gamma$ est ab\'elien, alors cela se reformule encore en
\begin{enumerate}[(i'')]
\addtocounter{enumi}{2}
\item\label{e_qua} $X$ s'identifie (apr\`es choix d'un point de $X$) \`a un groupe ab\'elien muni d'un morphisme d'image dense \'emanant de $\Gamma$.
\end{enumerate}
\end{fact}

\begin{proof}
(\ref{i_proj})$\Rightarrow$(\ref{i_odo}) est imm\'ediat. Faisons sa r\'eciproque: pour toute partie finie $I$ de $\Clo(X)$, on peut consid\'erer l'union $J(I)$ des $\Gamma$-orbites des \'el\'ements de $I$, qui est encore finie par hypoth\`ese et d\'efinir $X_I$ comme le quotient de $X$ par la partition d\'efinie par l'alg\`ebre bool\'eenne engendr\'ee par $J(I)$, c'est un $\Gamma$-ensemble, quotient de $X$. Clairement toute inclusion $I\subset I'$ induit une application quotient $X_{I'}\to X_I$. L'application naturelle de $X$ vers la limite projective des $X_I$ ($I$ parcourant les parties finies de $\Clo(X)$) a donc une image dense; comme $X$ est totalement s\'epar\'e, cette application est injective; par compacit\'e c'est donc un hom\'eo\-mor\-phisme.

Pour (\ref{i_odo})$\Rightarrow$(\ref{i_shift}), il suffit de montrer qu'un sous-d\'ecalage $X\subset A^\Gamma$ qui est un odom\`etre est n\'ecessairement fini. Par hypoth\`ese, il existe un sous-groupe d'indice fini $\Lambda$ de $\Gamma$ tel que pour tout $a\in A$, en notant $\Omega_a=\{w\in X:w(1)=a\}$, on a $g\Omega_a=\Omega_a$ pour tout $g\in\Lambda$. On en d\'eduit que $X$ est constitu\'e de suites $\Lambda$-p\'eriodiques \`a droite (et est donc fini): en effet, on voit dans un premier temps que $w(\lambda)=w(1)$ pour tout $\lambda\in\Lambda$ et $w\in X$. Puis si $g\in G$, comme $g^{-1}w\in X$, on a $g^{-1}w(\lambda)=g^{-1}w(1)$, autrement dit $w(g\lambda)=w(g)$.

R\'eciproquement, si $X$ n'est pas un odom\`etre, il existe un clouvert $\Omega$ dont la $\Gamma$-orbite est infinie. La construction du fait \ref{fait_sd} fournit une application \'equivariante continue vers le d\'ecalage $\{0,1\}^\Gamma$, d'image infinie.

L'assertion suppl\'ementaire (\ref{e_quo}') implique \'evidemment (\ref{i_proj}) en \'ecrivant $K$ comme intersection filtrante d\'ecroissante de sous-groupes ouverts d'indice fini. R\'eci\-pro\-quement, supposons que $X$ est limite projective filtrante de $\Gamma$-ensembles finis $X_i$. En fixant un point $w$ dont l'orbite est dense, on peut identifier chaque $X_i$ \`a un espace quotient $\hat{\Gamma}/\Lambda_i$ avec $\Lambda_i$ ouvert d'indice fini. L'application naturelle de $\hat{\Gamma}/\bigcap\Lambda_i$ vers la limite projective des $\hat{\Gamma}/\Lambda_i$ est continue, injective et d'image dense, donc est un hom\'eo\-mor\-phisme par compacit\'e; elle est $\Gamma$-\'equivariante.

Pour $\Gamma$ ab\'elien, pour obtenir (\ref{e_qua}'') \`a partir de (\ref{e_quo}') (le sens inverse \'etant trivial), on remarque que $K$ est distingu\'e et le groupe $\hat{\Gamma}/K$ est bien un groupe muni d'un morphisme d'image dense \'emanant de $\Gamma$. 
\end{proof}

\subsection{Semi-groupes pleins et groupes pleins}\label{defpl}Rappelons qu'une fonction {\em \'etag\'ee} entre deux ensembles est une fonction dont l'image est finie.

\begin{defn}
Soit $\Gamma$ un groupe discret agissant par auto\-hom\'eo\-mor\-phismes sur un espace topologique. On d\'efinit le {\em semi-groupe plein-topologique} associ\'e \`a $(\Gamma,X)$ comme
\[\lb\Gamma,X\rb\os=\{f:X\to X:\;\exists \kappa:X\to\Gamma\text{ continue \'etag\'ee }:\forall x\in X,\;f(x)=\kappa(x).x\}.
\]
On dit que la fonction \'etag\'ee $\kappa$ ci-dessus est {\em associ\'ee} \`a $f$.
\end{defn}

Notons que la pr\'ecision que $\kappa$ est \'etag\'ee est superflue lorsque $X$ est compact. Il est clair que $\lb\Gamma,X\rb\os$ est constitu\'e d'applications continues $X\to X$, contient l'identit\'e et est stable par composition, c'est donc bien un semi-groupe. Si l'action de $\Gamma$ sur $X$ est libre, ce qui sera souvent le cas par la suite, l'application $\kappa$ associ\'ee \`a un \'el\'ement $f$ est uniquement d\'etermin\'ee par $f$, ce qui donne une bijection canonique entre $\lb\Gamma,X\rb\os$ et l'ensemble $\mathcal{C}_\textnormal{et}(X,\Gamma)$ des fonctions continues \'etag\'ees $X\to\Gamma$. Ceci reste valable sous l'hypoth\`ese plus faible que tout \'el\'ement de $\Gamma$ distinct de l'identit\'e a un ensemble de points fixes d'int\'erieur vide dans $X$. 

Remarquons d\'ej\`a que si $X$ est compact m\'etrisable, et $\Gamma$ d\'enombrable (on consid\`ere les ensemble finis comme d\'enombrables), alors $\mathcal{C}_\textnormal{et}(X,\Gamma)$ est d\'enombrable, et en particulier $\lb\Gamma,X\rb\os$ est d\'enombrable.

\begin{defn}Le {\em groupe plein-topologique} de $(\Gamma,X)$ est d\'efini comme l'ensemble des \'el\'ements inversibles du semi-groupe $\lb\Gamma,X\rb\os$; on le note $\lb\Gamma,X\rb$.
\end{defn}
On peut remarquer que $$\lb\Gamma,X\rb=\lb\Gamma,X\rb\os\cap\Homeo(X).$$
En effet, on v\'erifie que pour $f\in \Homeo(X)$ donn\'e par $f(x)=\kappa(x).x$, on a $f^{-1}(x)=\penalty0\kappa(f^{-1}(x))^{-1}.x$ pour tout $x\in X$.

Lorsque $\Gamma\subset\Homeo(X)$ et $X$ est fix\'e, il est commode d'\'ecrire $\lb\Gamma\rb\os$ et $\lb\Gamma\rb$. Il est souvent possible de s'y ramener, en rempla\c cant $\Gamma$ par son image dans $\Homeo(X)$. Si $\Gamma$ est engendr\'e par un auto\-hom\'eo\-mor\-phisme $\vpi$, on \'ecrit alors $\lb\vpi\rb\os$ et $\lb\vpi\rb$. Dans une direction diff\'erente, le groupe $\Gamma$ peut \^etre fix\'e et l'espace $X$ sur lequel il agit variable, et on \'ecrit alors $\lb X\rb\os$ et $\lb X\rb$.

Si $\Gamma\subset\Homeo(X)$, on a toujours $\Gamma\subset\lb\Gamma\rb$. Remarquons que si $X$ est connexe, alors toute application continue $X\to\Gamma$ est constante, de sorte que $\lb\Gamma\rb\os=\lb\Gamma\rb=\Gamma$. Cette construction est donc surtout int\'eressante lorsque $X$ est totalement s\'epar\'e. 

Le groupe $\lb\Gamma,X\rb$ sera consid\'er\'e ici comme un groupe discret.
L'adjectif (quelque peu maladroit) {\og topologique\fg}  permet de distinguer $\lb\Gamma,X\rb$ du {\em groupe plein} $[\Gamma,X]\subset\Homeo(X)$ d\'efini comme le pr\'ec\'edent, mais sans supposer l'application $\kappa$ continue ou \'etag\'ee. Le groupe plein est souvent beaucoup plus gros que le groupe plein-topo\-logique. 

\begin{rema}\label{parado}
On dit que le groupe $\Gamma$ est paradoxal si, munissant $\Gamma$ de l'action \`a gauche sur lui-m\^eme, il existe dans $\lb\Gamma,\Gamma\rb\os$ deux injections d'images disjointes. Le th\'eor\`eme de Tarski (voir \cite{HS86}) dit que $\Gamma$ est paradoxal si et seulement si $\Gamma$ est non moyennable, et on peut alors m\^eme supposer que les images des deux injections pr\'ec\'edentes partitionnent $\Gamma$.
\end{rema}

On commence par quelques propri\'et\'es \'el\'ementaires et g\'en\'erales des groupes pleins-topologiques. En remarquant que les orbites de $\lb\Gamma\rb$ sont contenues dans les $\Gamma$-orbites, on obtient:

\begin{fact}
Si $\Gamma\subset\Homeo(X)$ est un groupe fini (ou, plus g\'en\'eralement, a des orbites de cardinal born\'e), alors les orbites de $[\Gamma]$ sont de cardinal born\'e, de sorte que le groupe plein $[\Gamma]$ (et donc $\lb\Gamma\rb$) est un groupe localement fini, d'exposant fini, car il se plonge dans un produit (infini) de groupes finis de cardinal born\'e.
\end{fact}

Notons que $X$ est totalement s\'epar\'e si et seulement si les clouverts s\'eparent les points. Dans ce cas, on a

\begin{fact}\label{fait_sn}
Si $X$ est totalement s\'epar\'e et si $\Gamma$ agit sur $X$ par auto\-hom\'eo\-mor\-phismes avec des orbites de cardinal non born\'e (par exemple, $\Gamma$ a une orbite infinie), alors $\lb\Gamma\rb$ contient, pour tout $n$, un sous-groupe isomorphe au groupe sym\'etrique $\mathfrak{S}_n$ (et contient m\^eme une copie du produit direct restreint $\bigoplus_{n\ge 1}\mathfrak{S}_n$). En particulier, le groupe $\lb\Gamma\rb$ ne satisfait aucune loi (un groupe $G$ satisfait une loi $w$, o\`u $w(t_1,\dots,t_k)$ est un \'el\'ement non trivial du groupe libre \`a $k$ g\'en\'erateurs, si $w(g_1,\dots,g_k)=1$ pour tous $g_1,\dots,g_k$ dans $G$).
\end{fact}
\begin{proof}
On consid\`ere des \'el\'ements distincts $x_1,\dots,x_n$ dans une m\^eme \mbox{$\Gamma$-orbite,} disons $x_i=g_ix_1$. On choisit un voisinage clouvert $U$ de $x_1$ tel que les $g_iU$ sont deux \`a deux disjoints. Si $\sigma\in\mathfrak{S}_n$ est une permutation de $\{1,\dots,n\}$, alors pour $x\in X$, on d\'efinit $T_\sigma(x)=g_{\sigma(i)}g_i^{-1}x\in g_{\sigma(i)}U$ si $x\in g_iU$, et $T_\sigma(x)=x$ si $x\notin\bigcup_ig_iU$. On v\'erifie imm\'ediatement que $\rho:\sigma\mapsto T_\sigma$ est un morphisme injectif de $\mathfrak{S}_n$ dans $\lb\Gamma\rb$.

Pour obtenir un plongement de $\bigoplus \mathfrak{S}_n$, on choisit simultan\'ement $x_i^{(n)}$, $1\le i\le n<\infty$ deux \`a deux distincts et, \`a $n$ fix\'e, dans la m\^eme orbite, et on choisit les $U_i^{(n)}$ \'egalement deux \`a deux disjoints. Comme les $\rho^{(n)}(\mathfrak{S}_n)$ ont des supports deux \`a deux disjoints, ils commutent deux \`a deux et engendrent leur produit direct restreint.

Pour la derni\`ere assertion, on utilise le fait qu'il n'existe aucune loi de groupe satisfaite simultan\'ement par tous les groupes sym\'etriques: en effet une telle loi serait satisfaite par tous les groupes finis, donc par les produits directs de groupes finis. Or les groupes libres, \'etant r\'esiduellement finis \cite[p.~170]{Sch}, se plongent dans de tels produits.
\end{proof}

La preuve d\'emontre en fait \'egalement

\begin{cor}\label{cor_trans}
Si $X$ est totalement s\'epar\'e, alors $\lb\Gamma\rb$ agit $\infty$-transitivement sur toute $\Gamma$-orbite, au sens o\`u il agit transitivement, pour tout $n$, sur les $n$-uplets d'\'el\'ements distincts. 
De plus, le sous-groupe d\'eriv\'e $\lb\Gamma\rb'$ agit $(n-2)$-transitivement sur toute \mbox{$\Gamma$-orbite} de cardinal au moins $n$, et donc $\infty$-transitivement sur toute orbite infinie.

En particulier, les $\lb\Gamma\rb'$-orbites co\"\i ncident avec les $\Gamma$-orbites (qui sont aussi les $\lb\Gamma\rb$-orbites), \`a l'exception des $\Gamma$-orbites \`a deux \'el\'ements.

En outre, si la r\'eunion des $\Gamma$-orbites de cardinal au plus 4 est d'int\'erieur vide, alors le commutant de $\lb\Gamma\rb'$ dans le groupe des permutations de $X$ (et donc dans $\Homeo(X)$) est trivial.
\end{cor}
\begin{proof}
En effet, on a v\'erifi\'e que pour toute partie $P$ \`a $n$ \'el\'ements de toute orbite, il existe une action du groupe sym\'etrique $\mathfrak{S}_n$, via $\lb\Gamma\rb$, qui induit l'action standard sur $P$. 

En particulier, son sous-groupe d\'eriv\'e, le groupe altern\'e $\mathfrak{A}_n$, agit via $\lb\Gamma\rb'$, qui agit donc $(n-2)$-transitivement.

Pour le commutant, on remarque que si on a trois points $x,gx,hx$ distincts dans une m\^eme $\Gamma$-orbite, alors pour tout voisinage clouvert $U$ de $x$ assez petit, on a un \'el\'ement de $\lb\Gamma\rb'$ permutant cycliquement les clouverts disjoints $U$, $gU$ et $hU$ et agissant par l'identit\'e ailleurs. En particulier, le commutant $C$ de $\lb\Gamma\rb'$ pr\'eserve $U\cup gU\cup hU$. Comme c'est vrai pour tout voisinage clouvert de $x$ inclus dans $U$, on en d\'eduit que $C$ laisse invariant $\{x,gx,hx\}$. Maintenant s'il existe deux autres points $g'x,h'x$ dans l'orbite, le m\^eme argument montre que $C$ laisse invariant $\{x,g'x,h'x\}$ et donc laisse invariant l'intersection de ces deux parties, \`a savoir $\{x\}$. Comme par hypoth\`ese l'ensemble des $x$ dont l'orbite a au moins 5 \'el\'ements est dense, cela prouve que $C$ est r\'eduit \`a l'identit\'e.
\end{proof}

Ce fait implique \'egalement, par exemple, le corolaire suivant.

\begin{cor}
Sous les hypoth\`eses du fait \ref{fait_sn}, $\lb\Gamma\rb$  n'admet aucune repr\'e\-sentation fid\`ele dans $\GL_d(A)$, o\`u $A$ est un anneau commutatif quelconque et $d$ un entier arbitraire. Il en est de m\^eme pour $\lb\Gamma\rb'$.
\end{cor}
\begin{proof}
On fixe $k$ assez grand, disons $k=k(d)=\max(5,4d^2)$. Montrons que pour tout anneau commutatif $A$, il n'y a aucun morphisme non trivial de $\Alt_{k(d)}$ vers $\GL_d(A)$. On remarque d'abord que le groupe altern\'e $\Alt_{k}$ contient $d^2$ paires d'\'el\'ements, chacune des paires ne commutant pas, mais commutant avec toutes les autres. Il en d\'ecoule ais\'ement (voir \cite{Abert}, \`a qui cette astuce est due) que tout morphisme injectif de $\Alt_{k}$ dans $\GL_m$ d'un corps envoie ces $2d^2$ \'el\'ements sur une famille lin\'eairement libre, impliquant en particulier que $m>d$. Donc, par simplicit\'e, tout morphisme de $\Alt_{k}$ dans $\GL_d$ d'un corps est trivial.

Comme tout anneau commutatif r\'eduit se plonge dans un produit de corps, on en d\'eduit que tout morphisme de $\Alt_{k}$ dans $\GL_d(A)$, pour $A$ anneau commutatif, a son image dans le noyau de $\GL_d(A)\to\GL_d(A/N)$, o\`u $N$ est le nilradical de $A$. Or ce noyau est un groupe localement nilpotent, puisqu'il est nilpotent quand $A$ est un anneau commutatif de type fini ou plus g\'en\'eralement noeth\'erien. Donc l'image du groupe simple non ab\'elien $\Alt_k$ dans $\GL_d(A)$ est en fait triviale.

Les sous-groupes $\Alt_{k}$ \'etant simples non ab\'eliens, ils sont \'evidemment inclus dans $\lb\Gamma\rb'$ et l'argument s'applique donc aussi \`a ce dernier.
\end{proof}

\begin{fact}\label{fpart}
On suppose qu'on a une partition de $X$ (en parties appel\'ees composantes). Si la partition est invariante par $\lb\Gamma\rb$, alors toute $\Gamma$-orbite est soit compos\'ee d'atomes (composantes r\'eduites \`a un point), soit incluse dans une seule composante. Sous l'hypoth\`ese plus faible que la partition est invariante par $\lb\Gamma\rb'$, la m\^eme assertion est vraie pour toute orbite de cardinal distinct de 2.
\end{fact}
\begin{proof}
On se donne une $\Gamma$-orbite non r\'eduite \`a un singleton (l'\'enonc\'e \'etant trivial pour un point fixe). Si par l'absurde elle ne v\'erifie pas la conclusion, elle contient deux points $x,y$ dans deux composantes distinctes, telle que la composante de $x$ contient un autre point $w$ de $X$.
Le fait \ref{fait_sn} fournit dans $\lb\Gamma\rb$ un \'el\'ement fixant $w$ et \'echangeant $y$ et $z$; cet \'el\'ement ne pr\'eserve pas la partition, contradiction.

Dans le cas d'une partition $\lb\Gamma\rb'$-invariante, le m\^eme raisonnement (pour une orbite \`a au moins 3 \'el\'ements) aboutit \`a l'existence de points $w,x,y$ v\'erifiant les m\^emes hypoth\`eses, et d'un point suppl\'ementaire $z$ dans l'orbite de $x$. Si $x\neq z$, le fait \ref{fait_sn} fournit dans $\lb\Gamma\rb'$ un \'el\'ement fixant $w$ et permutant cycliquement $x,y,z$; si $x=z$, il fournit dans $\lb\Gamma\rb'$ un \'el\'ement permutant cycliquement $w,x,y$; dans les deux cas, l'\'el\'ement d'ordre 3 obtenu ne pr\'eserve pas la partition.
\end{proof}

\section{Propri\'et\'es alg\'ebriques des groupes pleins-\penalty0 topologiques sur $\Z$}\label{pa}

\subsection{Simplicit\'e du groupe d\'eriv\'e}\label{simplicite}

On montre ici que si $\vpi$ est un auto\-hom\'eo\-mor\-phisme minimal d'un compact totalement s\'epar\'e $X$, alors le groupe d\'eriv\'e $\lb\vpi\rb'$ est simple. Ce r\'esultat est d\^u \`a Matui \cite{Ma1}. Une preuve plus directe en a \'et\'e donn\'ee par Bezuglyi et Medynets \cite{BeM}, que l'on va suivre. (On dira cependant plus loin quelques mots sur l'approche de Matui.)

La m\'ethode de Bezuglyi-Medynets consiste \`a montrer directement que pour tout $f\in\lb\vpi\rb$ non trivial, le commutateur de deux \'el\'ements \`a petit support (en un sens li\'e aux probabilit\'es $\vpi$-invariantes sur $X$) appartient au sous-groupe distingu\'e engendr\'e par $f$.
Le r\'esultat de simplicit\'e en d\'ecoule, une fois qu'on sait exprimer tout \'el\'ement comme produit d'\'el\'ements \`a petit support. Ceci utilise la notion de fonction de premier retour, d\'efinie ci-dessous.

Soit $X$ un espace topologique et soit $\psi$ un auto\-hom\'eo\-mor\-phisme. On dit que $A\subset X$ est $\psi$-{\em omniscient} si $\bigcup_{n\in\Z}\psi^n(A)=X$, ou autrement dit, si l'ensemble
$$Z(x)=Z(A,\psi,x)=\{n\in\Z:\psi^n(x)\in A\}$$
est non vide pour tout $x\in X$.

\begin{fact}
On suppose $X$ compact. Si $A\subset X$ est un ouvert $\psi$-omniscient, alors pour tout $x\in X$,
l'ensemble $Z(x)$ n'est ni minor\'e ni major\'e.
\end{fact}
\begin{proof}
Par hypoth\`ese, $Z(x)$ est non vide pour tout $x$. Soit $M_k$ l'ensemble des $x$ pour qui $\sup(Z(x))\le k$. C'est un ferm\'e; si on suppose par l'absurde qu'il est non vide pour au moins un $k$, il est en fait non vide pour tout $k$ puisque $\psi(M_k)=M_{k-1}$. En outre, $M_k\subset M_{k+1}$ pour tout $k$, donc par compacit\'e on obtient $\bigcap_k M_k\neq\emptyset$, ce qui est une contradiction. Ceci prouve que $Z(x)$ n'est jamais major\'e; pour la m\^eme raison il n'est jamais minor\'e.
\end{proof}

Ceci permet de d\'efinir, si $A$ est $\psi$-omniscient, la fonction 
$$\psi_A(x)=\psi^{k(x)}(x),\quad{\text{o\`u}}\quad k(x)=k_{A,\psi}(x) =	\left\{ \begin{array}{cl} \inf(Z(x)\cap \mathopen]0,+\infty\mathclose[) & \mbox{si } x\in A\,; \\ 0 & \mbox{sinon}.
\end{array}\right.$$

La fonction $\psi_A$ est appel\'ee {\em fonction de premier retour de $\psi$ sur} $A$; c'est un auto\-hom\'eo\-mor\-phisme, de r\'eciproque $(\psi^{-1})_A$. Si $A$ est clouvert, il est imm\'ediat que la fonction $k$ est continue (c'est-\`a-dire localement constante), si bien que $\psi_A\in\lb\psi\rb$ et donc $\psi_A\in\lb\vpi\rb$ si $\psi\in\lb\vpi\rb$ pour un auto\-hom\'eo\-mor\-phisme $\vpi$ pr\'ealablement fix\'e.

\begin{fact}\label{fact_pe}
Si $X$ est compact, $\psi$ est un auto\-hom\'eo\-mor\-phisme, et $A$ est un clouvert $\psi$-omniscient, alors $\psi^{-1}\psi_A$ est p\'eriodique.
\end{fact}
\begin{proof}
Posons $v(x)=\inf(Z(x)\cap\penalty0\mathopen]0,+\infty\mathclose[)$ et $u(x)=\sup(Z(x)\cap\penalty0\mathopen]{-\infty},0])$. Remarquons que $$1\le v-u\le\sup k\quad\text{et}\quad(\psi^{-1}\psi_A)^{v(x)-u(x)}x=x$$ pour tout $x$. Donc $(\psi^{-1}\psi_A)^n=1$, o\`u $n=(\sup k)!$.
\end{proof}

\begin{lem}\label{l:domaine}
Soit $f$ un auto\-hom\'eo\-mor\-phisme d'un espace compact totalement s\'epar\'e $X$, et $n\ge 1$ un entier. On suppose que $f^i$ est sans point fixe pour tout $1\le i\le n-1$. Alors il existe un clouvert $U$ tel que les $(f^i(U))$, pour $0\le i\le n-1$, sont deux \`a deux disjoints et $\bigcup_{i\in\Z}f^i(U)=X$.

En particulier, si $f^n=1$ et si $f$ d\'efinit une action libre de $\Z/n\Z$, alors les $(f^i(U))$, pour $0\le i\le n-1$, forment une partition de $X$. 
\end{lem}
\begin{proof}
Pour tout $x\in X$ on choisit un voisinage clouvert $V_x$ disjoint de $\bigcup_{1\le i\le n-1}f^i(V_x)$. On extrait du recouvrement $(V_x)$ un recouvrement fini $(W_k)_{1\le k\le m}$. On pose, pour $1\le k\le m$, $$X_k=W_k\smallsetminus \bigcup_{0<j<k,\;-n< i< n}f^i(W_j).$$ Alors $U=\bigcup_{1\le k\le m}X_k$ satisfait la condition demand\'ee.
\end{proof}

Rappelons que par moyennabilit\'e de $\Z$ (se refl\'etant ici par le th\'eor\`eme de Kakutani), tout auto\-hom\'eo\-mor\-phisme $\vpi$ d'un compact non vide $X$ pr\'eserve une probabilit\'e sur les bor\'eliens. Notons $\mathcal{M}_\vpi(X)$ l'ensemble des probabilit\'es bor\'eliennes $\vpi$-invariantes sur $X$; c'est un compact pour la topologie faible-* (convergence \'etroite). Si $B$ est un bor\'elien de $X$, notons $$L^+_\vpi(U)=\sup_{\mu\in \mathcal{M}_\vpi(X)}\mu(B);\quad L^-_\vpi(U)=\inf_{\mu\in \mathcal{M}_\vpi(X)}\mu(B).$$

\begin{lem}\label{decompo}
Soit $\vpi$ un auto\-hom\'eo\-mor\-phisme sans orbite finie du compact totalement s\'epar\'e $X$. Alors pour tout $\psi\in\lb\vpi\rb$ et $\eps>0$, il existe $m$ et $\psi_1,\dots,\psi_m$ dans $\lb\vpi\rb$ tels que $\psi=\psi_1\dots \psi_m$ et $L^+_\vpi(\Supp(g_i))<\eps$ pour tout $i$. Si de plus $\psi^n=1$ pour un entier $n\ge 1$, on peut s'arranger pour que $\psi_i^n=1$ pour tout $i=1,\dots,m$. 
\end{lem}
\begin{proof}
Soit $o_\psi(x)$ le cardinal de la $\langle\psi\rangle$-orbite de $x$, et posons $X_n=X_n(\psi)=\{x\mid o_\psi(x)=n\}$ et $X_{\le n}=\bigcup_{k\le n}X_k$.

Remarquons que comme $\vpi$ n'a pas d'orbite finie, tout \'el\'ement $x\in X$ poss\`ede des voisinages clouverts avec $L^+_\vpi$ arbitrairement petit. Si $\psi^n=1$ pour un entier $n\ge 1$, il en d\'ecoule que tout \'el\'ement $x$ poss\`ede des voisinages $\psi$-invariants avec $L^+_\vpi$ arbitrairement petit. On en d\'eduit, si $\psi$ est p\'eriodique, l'existence d'une partition $(U_j)$ en clouverts $\psi$-invariants avec $L^+_\vpi(U_j)<\eps$ pour tout $j$, qui induit une d\'ecomposition $\psi=\prod\psi_j$, o\`u $\psi_j$ a un support inclus dans $U_j$. Le lemme est donc d\'emontr\'e pour $\psi$ p\'eriodique, et de plus $\psi_j^n=1$ pour tout $j$.

Faisons maintenant le cas oppos\'e o\`u $\psi$ n'a aucune orbite finie. On choisit un entier 
\mbox{$n>1/\eps$.} 
Le lemme \ref{l:domaine} montre qu'il existe un clouvert $U$ $\psi$-omniscient tel que 
\mbox{$U\cap \psi^i(U)=\emptyset$} pour tout $0\le i\le n-1$. Donc, comme toute probabilit\'e bor\'elienne $\vpi$-invariante est aussi $\psi$-invariante, on a $L^+_\vpi(U)\le 1/n<\eps$. On \'ecrit
$\psi=\psi_U\tau$, o\`u $\tau=\psi_U^{-1}\psi$ est p\'eriodique par le fait \ref{fact_pe}. Alors $L_\vpi^+(\Supp(\psi_U))<\eps$, et on peut d\'ecomposer $\tau$ par le cas p\'eriodique. 

Pour $\psi$ quelconque, si $N$ est assez grand, alors $\psi\vpi^N$ et $\vpi^{-N}$ n'ont pas d'orbite finie et $\psi=(\psi\vpi^N)\vpi^{-N}$, si bien qu'on peut d\'ecomposer chacun des deux par le cas pr\'ec\'edent.
\end{proof}

On utilise le lemme suivant, d\^u \`a Glasner et Weiss \cite{GW}. Sa preuve se distingue ici par l'invocation d'arguments \`a caract\`ere ergodique.

\begin{lem}\label{l:gw}
Soit $\vpi$ un auto\-hom\'eo\-mor\-phisme minimal d'un compact totalement s\'epar\'e $X$. Soient $A,B$ des clouverts tels que $\mu(B)<\mu(A)$ pour toute probabilit\'e bor\'elienne $\vpi$-invariante sur $X$. Alors il existe $\alpha\in\lb\vpi\rb'$ tel que $\alpha(B)\subset A$.
\end{lem}
\begin{proof}
Le cas o\`u $X$ est fini \'etant trivial, on suppose $X$ infini, si bien que l'action de $\Z$ est libre. On pose $f=1_A-1_B$; on a donc $\int fd\mu>0$ pour tout $\mu\in \mathcal{M}_\vpi(X)$. On commence par observer qu'il existe $c>0$ et $N_0\ge 5$ tels que pour tout $N\ge N_0$ et tout $x\in X$, on a
\begin{equation}\frac{1}{N}\sum_{i=0}^{N-1}f(\vpi^i(x))\ge c.\label{esti}\end{equation}
Sinon, on trouve une suite $(N_k)$ tendant vers l'infini, une suite $(c_k)$ avec $\overline{\lim}\,c_k\le 0$ et une suite $(x_k)$ telle que $\frac{1}{N_k}\sum_{i=0}^{N_k-1}f(\vpi^i(x_k))\le c_k$. Autrement dit, en posant $\mu_k=\frac{1}{N_k}\sum_{i=0}^{N_k-1}\vpi^i_\ast\delta_{x_k}$, on a $\int fd\mu_k\le c_k$. Par compacit\'e, la suite $(\mu_k)$ admet un point d'accumulation $\mu$ pour la topologie faible-*, et donc $\int fd\mu\le 0$. Or $\mu$ doit \^etre $\vpi$-invariante et c'est une contradiction.

Consid\'erons maintenant un clouvert non vide $U$ inclus dans $A$, tel que les $\vpi^i(U)$ pour $0\le i<N_0$ soient deux \`a deux disjoints. Comme $\vpi$ est minimal, $U$ est $\vpi$-omniscient. Soient $k=k_{A,\vpi}$ la fonction temps de retour sur $U$ et $\kappa$ une borne sup\'erieure pour $k$. On peut trouver une partition en clouverts $(U'_\ell)$ de $U$ telle que $k$ soit constante sur chacune des composantes de la partition. Remarquons que les $\vpi^i(U'_\ell)$, pour $0\le i\le k(U'_\ell)-1$, sont deux \`a deux disjoints, et qu'en faisant varier \'egalement $i$ ils forment une partition de $U$ (commun\'ement appel\'ee partition en {\og tours de Kakutani-Rokhlin\fg}). On raffine la partition $(U'_\ell)$ en une partition de $U$ en clouverts $(U_j)$ et telle que chaque $\vpi^i(U_j)$, pour $0\le i\le k(U_j)-1$, soit inclus dans l'un des ensembles $A\smallsetminus B$, $B\smallsetminus A$, $A\cap B$, ou bien $X\smallsetminus (A\cup B)$. Comme pour la partition pr\'ec\'edente on a une partition
$$X=\bigsqcup_j\bigsqcup_{0\le i\le k(U_j)-1}\vpi^i(U_j).$$
Pour chaque $j$, appelons {\og tour\fg}  une r\'eunion $\bigsqcup_{0\le i\le k(U_j)-1}\vpi^i(U_j)$. Pour chaque $j$, soit $A_j$ (resp.\ $B_j$, $C_j$, $D_j$) l'ensemble des $i\in\{0,\dots, k(U_j)-1\}$ tel que $\vpi^i(U_j)$ appartient \`a $A\smallsetminus B$
(resp.\ $B\smallsetminus A$, $A\cap B$, $X\smallsetminus (A\cup B)$). Alors pour tout $j$ on a $\#(A_j)\ge \#(B_j)$. En effet, par (\ref{esti}), pour $x\in U_j$, on a 
$$\frac{\#(A_j)-\#(B_j)}{k(U_j)}=\frac{1}{k(U_j)}\sum_{i=0}^{k(U_j)-1}f(\vpi^i(x))\ge c>0.$$
Pour tout $j$, on peut trouver une permutation $\sigma_j$ de $\{0,\dots, k(U_j)-1\}$ qui envoie $B_j$ dans $A_j$, et pr\'eserve $C_j$ et $D_j$. Cela d\'efinit naturellement un \'el\'ement $\alpha\in\lb\vpi\rb$, pr\'eservant globalement chaque tour, qui envoie $B$ dans $A$. Remarquons qu'on peut s'arranger pour que chaque $\sigma_j$ soit une permutation paire: en effet comme $k(U_j)\ge 5$ pour tout $j$, au moins l'un des ensembles $A_j,\dots,D_j$ a deux \'el\'ements et on peut donc composer par une transposition si n\'ecessaire. Ce choix \'etant fait, l'\'el\'ement $\alpha$ obtenu est un commutateur.
\end{proof}


\begin{thm}\label{dersim}
Soit $\vpi$ un auto\-hom\'eo\-mor\-phisme minimal d'un compact totalement s\'epar\'e $X$ \`a au moins 5 points. Alors le groupe d\'eriv\'e $\lb\vpi\rb'$ est simple et est contenu dans tout sous-groupe distingu\'e non trivial de $\lb\vpi\rb$.
\end{thm}
\begin{proof}
Si $X$ est fini, ce n'est rien d'autre que la simplicit\'e du groupe altern\'e, qu'il n'est pas n\'ecessaire de red\'emontrer ici. On suppose maintenant $X$ infini.


Consid\'erons un sous-groupe non trivial $N$ de $\lb\vpi\rb$ normalis\'e par $\lb\vpi\rb'$. Le commutant de $\lb\vpi\rb'$ dans $\Homeo(X)$ est trivial (corolaire \ref{cor_trans}), si bien qu'on peut choisir un \'el\'ement non trivial $f$ dans $N\cap\lb\vpi\rb'$.

Pour d\'emontrer le r\'esultat, il faut montrer que pour tous $g,h\in\lb\vpi\rb$ on a $[g,h]\in N$. Supposons, dans un premier temps\footnote{Pierre Py m'a signal\'e une erreur dans la preuve dans une version pr\'ec\'edente de ce texte, reprenant une erreur dans \cite{BeM}. L'argument donn\'e ne fonctionnait en effet qu'en supposant $N$ distingu\'e dans $\lb\vpi\rb$, ce qui ne suffisait pas \`a \'etablir la simplicit\'e de $\lb\vpi\rb'$.}, 
que $N$ est distingu\'e dans $\lb\vpi\rb$.

Soit $E$ un clouvert non vide disjoint de $f(E)$. L'application qui \`a $\mu\in\mathcal{M}_\vpi(X)$ associe $\mu(E)$ est continue; par compacit\'e de $\mathcal{M}_\vpi(X)$, il en d\'ecoule que $\eps=L^-_\vpi(E)>0$.

Par le lemme \ref{decompo}, on peut \'ecrire $g=\prod g_i$ et $h=\prod h_j$ avec
$L^+_\vpi(\Supp(g_i))$ et $L^+_\vpi(\Supp(h_j))$ strictement inf\'erieurs \`a $\eps/2$. L'\'el\'ement $[g,h]$ appartient au sous-groupe distingu\'e engendr\'e par les $[g_i,h_j]$ (comme on voit en quotientant par ce dernier). Donc il suffit de montrer que les $[g_i,h_j]$ sont dans $N$.


Soit $U$ le clouvert $\Supp(g_i)\cup\Supp(h_j)$; on a $L^+_\vpi(U)<\eps=L^-_\vpi(E)$.
Par le lemme \ref{l:gw}, il existe un \'el\'ement $\alpha\in\lb\vpi\rb'$ tel que $\alpha(U)\subset E$; soit $q=\alpha^{-1}f\alpha\in N$. L'\'el\'ement $\hat{h}_j=[h_j,q]=h_jqh_j^{-1}q^{-1}$ appartient \`a $N$, et donc $[g_i,\hat{h}_j]\in N$. Puisque $q(U)\cap U=\emptyset$, les \'el\'ements $g_i^{-1}$ et $qh_jq^{-1}$ commutent, donc 
$$[g_i,\hat{h}_j]=g_ih_j(qh_j^{-1}q^{-1})g_i^{-1}(qh_jq^{-1})h_j^{-1}=g_ih_jg_i^{-1}h_j^{-1}=[g_i,h_j],$$ 
si bien que $[g_i,h_j]\in N$.

On a donc d\'emontr\'e que tout sous-groupe distingu\'e de $\lb\vpi\rb$ non r\'eduit \`a l'identit\'e contient $\lb\vpi\rb'$. On peut appliquer ce r\'esultat  au sous-groupe de $\lb\vpi\rb$ engendr\'e par les \'el\'ements d'ordre 5. En effet, \'etant donn\'e un tel \'element, l'ensemble de ses points fixes est clouvert, et on peut appliquer le fait \ref{l:domaine} au compl\'ementaire de l'ensemble de ses points fixes, ce qui permet de v\'erifier que cet \'el\'ement est contenu dans un sous-groupe de $\lb\vpi\rb$ isomorphe au groupe altern\'e $\Alt_5$; en particulier, tout \'el\'ement d'ordre 5 est un commutateur, et d'autre part le fait \ref{fait_sn} montre qu'il existe des \'el\'ements d'ordre 5; ceci d\'emontre que le groupe $\lb\vpi\rb'$ est \'egal au sous-groupe de $\lb\vpi\rb$ engendr\'e par les \'elements d'ordre 5 (le lecteur peut v\'erifier qu'il en est de m\^eme pour tout ordre entier impair sup\'erieur \`a~2). Ceci d\'emontre en outre que $\lb\vpi\rb'$ est engendr\'e par ses sous-groupes isomorphes \`a $\Alt_5$, et donc co\"incide avec son sous-groupe d\'eriv\'e (ce qu'on voit \'egalement en remarquant que $\lb\vpi\rb''$ est non trivial et distingu\'e dans $\lb\vpi\rb$ et contient donc le groupe d\'eriv\'e).


Pour montrer le th\'eor\`eme, puisque $\lb\vpi\rb''=\lb\vpi\rb'$, il suffit de montrer que pour tous $g,h\in\lb\vpi\rb'$ on a $[g,h]\in N$. On peut donc, avec la version {\og p\'eriodique d'ordre 5\fg}   du lemme \ref{decompo}, faire le raisonnement pr\'ec\'edent en supposant que $g_i,h_j\in\lb\vpi\rb'$, auquel cas il suffit de supposer $N$ distingu\'e dans $\lb\vpi\rb'$ pour obtenir que $[g_i,h_j]$ appartient \`a $N$. 
\end{proof}




L'approche de Matui, tr\`es diff\'erente, est plus complexe mais apporte d'autres informations. Elle suppose toujours de mani\`ere essentielle que $\Gamma=\lb\vpi\rb$ et $X$ est compact totalement s\'epar\'e, minimal. Elle est bas\'ee sur une \'etude du
sous-groupe $$\lb\vpi\rb_{[x\rangle}=\{f\in\lb\vpi\rb:f(\vpi^\N(x))=\vpi^\N(x)\},\quad\text{o\`u}\quad\vpi^\N(x)=\{\vpi^n(x):n\in\N\};$$
montrant notamment que si $x,y$ appartiennent \`a des orbites distinctes alors $\lb\vpi\rb'\subset\lb\vpi\rb_{[x\rangle}\lb\vpi\rb_{[y\rangle}$.
On utilisera ce sous-groupe \`a plusieurs reprises par la suite (\S\ref{abelianise} et \S\ref{lecri}), mais de mani\`ere moins approfondie que dans \cite{Ma1}.

\subsection{Type-finitude}\label{type}

On commence par l'observation suivante qui justifie que sous des hypoth\`eses tr\`es g\'en\'erales, \^etre un sous-d\'ecalage est une condition n\'eces\-saire pour que le groupe plein-topologique et son d\'eriv\'e soient de type fini.

\begin{propp}\label{sdnec}
Soit $\Gamma$ un groupe agissant par auto\-hom\'eo\-mor\-phismes sur un espace compact totalement s\'epar\'e $X$, sans point fixe global. Si $\Lambda=\lb\Gamma\rb$ est de type fini et $X$ n'a qu'un nombre fini de $\Gamma$-orbites de cardinal au plus deux, alors $(\Gamma,X)$ est un sous-d\'ecalage. Si de plus aucune $\Gamma$-orbite n'est de cardinal deux, il en est de m\^eme si $\Lambda$ est un sous-groupe de $\lb\Gamma\rb$ contenant $\lb\Gamma\rb'$.
\end{propp}
\begin{proof}
Pour tout $x$ il existe $\gamma_x$ dans $\Gamma$ et un voisinage clouvert $V_x$ de $x$ disjoint de $\gamma_xV_x$. On peut raffiner et extraire une partition $(U_j)$ en clouverts  telle que chaque $U_j$ ne contient aucune orbite.

Supposons que $\lb\Gamma\rb$ est de type fini. Consid\'erons un nombre fini de clouverts sur qui chaque g\'en\'erateur est donn\'e par action d'un \'el\'ement donn\'e de $\Gamma$. Consid\'erons l'alg\`ebre bool\'eenne $\Gamma$-invariante engendr\'ee par ces clouverts et par les $U_j$. Elle d\'efinit une partition $\Gamma$-invariante de $X$. Comme les $U_j$ ne contiennent aucune orbite, aucune composante de cette partition ne contient d'orbite. Donc, par le fait \ref{fpart}, toute $\Gamma$-orbite est compos\'ee d'atomes de cette partition, donc c'est la partition triviale, ce qui montre, par le fait \ref{fait_sd}, que $(\Gamma,X)$ est un sous-d\'ecalage.

Pour $\Lambda$ contenant $\lb\Gamma\rb'$, gr\^ace au fait \ref{fpart}, l'hypoth\`ese suppl\'ementaire permet d'appliquer le m\^eme raisonnement.
\end{proof}

On va maintenant se restreindre au cas de $\Gamma=\Z$. Matui prouve le th\'eor\`eme suivant \cite[Theorem 5.4]{Ma1}.

\begin{thm}\label{fg}
Soit $(X,\vpi)$ un sous-d\'ecalage minimal sur $\Z$. Alors le groupe d\'eriv\'e $\lb\vpi\rb'$ du groupe plein-topologique correspondant est de type fini.
\end{thm}

On commence par deux lemmes. Rappelons qu'un sous-d\'ecalage $X\subset A^\Z$ est dit $k$-{\em propre} si pour tout $w\in X$ et tous $m,n$ tels que $0<|m-n|\le k$ on a $w(m)\neq w(n)$. 

\begin{lem}\label{4p}
Soit $d\ge 1$. Tout sous-d\'ecalage $(X,\vpi)$ sur $\Z$ sans orbite de cardinal au plus $d$ est isomorphe \`a un sous-d\'ecalage $d$-propre (sur un autre alphabet).
\end{lem}
\begin{proof}
Par hypoth\`ese (cf. le fait \ref{fait_sd}), $X$ est compact totalement s\'epar\'e et il existe une partition de $X$ en clouverts, index\'ee par un ensemble fini $A$, dont les $\Gamma$-translat\'es s\'eparent les points. Notons $\vpi$ l'auto\-hom\'eo\-mor\-phisme de d\'ecalage. Par hypoth\`ese, pour tout $x\in X$, les \'el\'ements $x,\vpi(x),\dots,\vpi^d(x)$ sont deux \`a deux distincts. Donc on peut raffiner la partition en une partition en clouverts $(V_b)_{b\in B}$, index\'ee par un ensemble fini $B$ plus grand, tel que pour tout $b$, les $\vpi^i(V_b)$, pour $i=0,\dots,d$, sont deux \`a deux disjoints. La r\'ealisation correspondante comme sous-d\'ecalage sur $B^\Gamma$ est donc, par d\'efinition, $d$-propre.
\end{proof}

\begin{defn}Si $I$ est une partie finie de $\Z$, on dit qu'un clouvert $U$ est $I$-{\em bon} si les $(\vpi^i(U))_{i\in I}$ sont deux \`a deux disjoints. On abr\`ege {\og$\{-1,0,1\}$-bon\fg}  en {\og bon\fg}. \'etant donn\'e un clouvert $U\subset X$ bon, on d\'efinit $\sigma_U\in\lb\vpi\rb$ comme l'\'el\'ement d'ordre au plus 3 qui vaut l'identit\'e hors de $\vpi^{-1}U\cup U\cup\vpi U$ et permute cycliquement:
$$U\stackrel{\vpi}\longrightarrow\vpi(U)\stackrel{\vpi^{-2}}\longrightarrow\vpi^{-1}(U)\stackrel{\vpi}\longrightarrow U.$$\label{d_bon}
\end{defn}

\begin{lem}\label{eng_cb}
Soit $(X,\vpi)$ un sous-d\'ecalage minimal infini sur $\Z$.
Le groupe $\lb\vpi\rb'$ est engendr\'e par les $\sigma_U$, o\`u $U$ parcourt les clouverts bons.
\end{lem}
\begin{proof}
Commen\c cons par observer que $\lb\vpi\rb'$ est engendr\'e par ses \'el\'ements d'ordre 3: en effet il existe dans $\lb\vpi\rb'$ des \'el\'ements d'ordre 3, et comme $\lb\vpi\rb'$ est simple (th\'eor\`eme \ref{dersim}), ils forment donc une partie g\'en\'eratrice. Disons qu'un \'el\'ement $\sigma$ d'ordre~3 est {\em sp\'ecial} s'il existe une partie \`a 3 \'el\'ements $I=\{i,j,k\}$ et un clouvert $U$ $I$-bon tel que $\sigma$ est l'identit\'e hors de $\bigcup_{i\in I}\vpi^i(U)$ et \'echange $\vpi^i(U)$, $\vpi^j(U)$ et $\vpi^k(U)$ par action de $\vpi^{j-i}$, $\vpi^{j-k}$, et $\vpi^{k-i}$. Si $\sigma$ est un \'el\'ement d'ordre 3 quelconque et $\kappa$ est la fonction \'etag\'ee associ\'ee, alors en regardant les fibres de la fonction $x\mapsto (\kappa(x),\kappa(\vpi(x)),\kappa(\vpi^2(x)))$, on voit que $\sigma$ est un produit d'\'el\'ements d'ordre 3 sp\'eciaux \`a supports disjoints. Soit maintenant $\sigma$ d'ordre 3 sp\'ecial, et $U$ et $I$ les parties correspondantes. Soit $J$ un segment entier contenant $I$. Alors il existe une partition $(U_k)$ de $U$ en clouverts $J$-bons; on peut \'ecrire $\sigma$ comme produit de ses restrictions aux $\bigcup_{i\in I}\vpi^i(U_k)$; ainsi on est ramen\'e au cas o\`u $\sigma$ a la propri\'et\'e suppl\'ementaire que $I$ est $J$-bon. On observe alors que pour tout intervalle entier $J$, le groupe altern\'e est engendr\'e par les 3-cycles $(i-1;i;i+1)$ compris dans $J$ (voir par exemple \cite[Lemma 5.1]{Ma1} pour une preuve). Or ce 3-cycle correspond \`a l'\'el\'ement $\sigma_{\vpi^i(U)}$; ainsi $\sigma$ est produit de tels \'el\'ements et de leurs inverses, ce qui termine la preuve.
\end{proof}

\begin{proof}[Preuve du th\'eor\`eme \ref{fg}]

On suppose $X$ infini, puisque sinon le r\'esultat est trivial. Par le lemme \ref{4p}, on peut supposer que $X\subset A^\Z$ est un sous-d\'ecalage 4-propre.

On commence par l'observation suivante: si $U,V\subset X$ sont des clouverts et si les ensembles $\vpi^{-1}U$, $U$, $\vpi U$, $\vpi^{-1}V$, $V$ et $\vpi V$ sont deux \`a deux disjoints {\em sauf} peut-\^etre $\vpi U$ et $\vpi^{-1}V$, alors, en notant $[s,t]=sts^{-1}t^{-1}$, on a l'\'egalit\'e
\begin{equation}\label{qeqz}[\sigma_V, \sigma_U^{-1}]=\sigma_{\vpi U\cap \vpi^{-1}V}.\end{equation}

Si $I$ est une partie finie non vide de $\Z$ et $f\in A^I$, on d\'efinit le {\em cylindre}
$$\Cyl(I,f)=\{w\in X\mid w|_I=f\}.$$
Remarquons que les cylindres sont des clouverts, et sont bons (d\'efinition \ref{d_bon}) car $X$ est 1-propre (et $I$ non vide); en outre, ils forment, en faisant varier $I$ et $f$, une base de la topologie de $X$.

On d\'efinit $\Lambda$ comme le sous-groupe de $\lb\Gamma\rb$ engendr\'e par l'ensemble fini constitu\'e des \'el\'ements de la forme $\sigma_{\Cyl(\{-1,0,1\},f)}$ o\`u $f\in A^{\{-1,0,1\}}$. On va prouver le th\'eor\`eme en montrant que $\Lambda=\lb\vpi\rb'$. Par le lemme \ref{eng_cb}, il suffit de v\'erifier que les $\sigma_U$, pour $U$ clouvert bon, sont dans $\Lambda$. En d\'ecomposant $U$, on voit qu'il suffit de le v\'erifier dans le cas o\`u $U$ est un cylindre $\Cyl(I_n,h)$, o\`u $I_n=\{-n,\dots,n\}$.

V\'erifions donc par r\'ecurrence sur $n\ge 1$ que pour tout $h\in A^{I_n}$ on a $\sigma_{\Cyl(I_n,h)}\in\Lambda$. C'est vrai par hypoth\`ese pour $n=1$ et supposons que $n\ge 2$ est que c'est d\'emontr\'e en de\c c\`a.
On note $\tau(n)=n+1$ la translation sur $\Z$, si bien que $\tau$ agit sur les fonctions partiellement d\'efinies sur $\Z$: si on a $f\in A^I$, alors $\tau\cdot f\in A^{\tau I}$ est d\'efinie par 
\mbox{$\tau\cdot f(n)=f(n-1)$} (on \'ecrit $\tau I=\tau(I)$ pour all\'eger les notations).
La fonction $\tau^{\pm 1}\cdot h$ est d\'efinie sur $\tau^{\pm 1}I_n$. Soit $h_\pm$ sa restriction \`a $I_{n-1}$. Notons $Y_\pm=\Cyl(I_{n-1},h_\pm)$. Alors 
on a $$\tau Y_-\cup\tau^{-1}Y_+\subset \Cyl(\{0\},h(0)).$$ Donc, par 4-propret\'e, on a
$$Y_+\cap \tau^i Y_+=Y_-\cap \tau^i Y_-=\emptyset\quad\text{et}\quad\tau^{j+1} Y_-\cap\tau^{k-1}Y_+=\emptyset$$
pour tous $i,j,k\in \Z$ tels que $1\le |i|\le 4$ et $1\le|j-k|\le 4$.
Les hypoth\`eses de (\ref{qeqz}) sont donc remplies, et on d\'eduit que 
$$\sigma_{\Cyl(I_n,h)}=\left[\sigma_{\Cyl(I_{n-1},h_-)}\,,\,\sigma_{\Cyl(I_{n-1},h_+)}^{-1}\right]$$
et par cons\'equent $\sigma_{\Cyl(I_n,h)}\in\Lambda$.
\end{proof}

\begin{rema}
Dans le meilleur des cas, $A$ a 6 \'el\'ements (\`a cause de la 4-propret\'e), cela fait une partie g\'en\'eratrice \`a $120=6\times 5\times 4$ \'el\'ements (ou $241=2\times 120+1$ si on veut une partie sym\'etrique avec 1), soit autant que d'injections de $\{-1,0,1\}$ dans $A$ ; c'est loin d'\^etre optimal; dans un cas particulier Matui \cite[\S 6]{Ma1} donne une partie g\'en\'eratrice \`a 4 \'el\'ements (voir aussi \cite[\S 3.1]{Ma2}).
\end{rema}

\subsection{Structure de l'ab\'elianis\'e}\label{abelianise}

Les r\'esultats de cette sous-partie sont principalement dus \`a 
Giordano-Putnam-Skau et Matui \cite{GPS,Ma1}\footnote{{\it (Ajout f\'evrier 
2020)} Voir aussi l'appendice \ref{ajout}}.

Soit $X$ un compact totalement s\'epar\'e non vide et $\vpi$ un auto\-hom\'eo\-mor\-phisme. Soit $x$ un point dont la $\vpi$-orbite est infinie. On note $\N=\{0,1,2\dots\}$ et $\N^c$ son compl\'ementaire dans $\Z$. On d\'efinit
\begin{eqnarray*}
\mo_x : \lb\vpi\rb & \to & \Z\\
\psi & \mapsto & \#\left(\vpi^{\N^c}(x)\cap\psi^{-1}(\vpi^{\N}(x))\right)-\#\left(\vpi^\N(x)\cap\psi^{-1}(\vpi^{\N^c}(x))\right)
\end{eqnarray*}
Intuitivement, c'est le {\og transfert global\fg}  d'\'el\'ements de gauche \`a droite dans l'orbite de $x$.

\begin{propp}\label{modp}L'application $\mo_x$ est un morphisme de groupes et ne d\'epend pas du choix de $x$ dans son orbite; on a $\mo_x(\vpi)=1$ (si bien que $\mo_x$ est surjectif). De plus, si $\Hom(\lb \vpi\rb,\Z)$ est muni de la topologie de la convergence ponctuelle, l'application $x\mapsto\mo_x$, d\'efinie sur la r\'eunion des orbites infinies, est continue. 
\end{propp}
\begin{proof}
Que $\mo_x$ est un morphisme est un fait g\'en\'eral sur le groupe des permutations \`a d\'eplacement born\'e de $\Z$ et est laiss\'e en exercice au lecteur. Le fait que $\mo_x=\mo_{\vpi(x)}$ en d\'ecoule, car $\mo_{\vpi(x)}(\psi)=\mo_x(\vpi^{-1}\circ\psi\circ\vpi)$ pour tout $\psi\in\lb\vpi\rb$.

V\'erifions la deuxi\`eme assertion. Il faut montrer que si $(x_n)$ converge vers $x$ et $\psi$ est fix\'e alors $\mo_{x_n}(\psi)$ est \'egal, pour $n$ assez grand, \`a $\mo_x(\psi)$. Soit $m$ une borne sup\'erieure sur la fonction \'etag\'ee associ\'ee \`a $\psi$. Soit $(U_i)$ une partition en clouverts trivialisant $\psi$. Il existe $n_0$ tel que pour tout $|i|\le m$, l'\'el\'ement $\vpi^i(x_n)$ est dans la m\^eme composante de la partition $(U_i)$ que $x$. On en d\'eduit facilement que $\mo_{x_n}(\psi)=\mo_x(\psi)$ pour tout $n\ge n_0$.
\end{proof}

En particulier, s'il existe une $\vpi$-orbite dense (par exemple, $X$ est minimal), alors $\mo_x$ ne d\'epend pas de $x$, on le note alors $\mo$.

\begin{thm}\label{tab}
Si $X$ est un compact totalement s\'epar\'e et $\vpi$ agit minimalement, alors le noyau $\lb\vpi\rb^0=\Ker(\mo)$ est engendr\'e par ses \'el\'ements d'ordre fini.
\end{thm}

Cela implique en particulier que $\mo$ est l'unique morphisme de $\lb\vpi\rb$ vers $\Z$ envoyant $\vpi$ sur 1. Or, il existe d'autres mani\`eres de construire ce morphisme. Par exemple, si $\mu$ est une probabilit\'e bor\'elienne invariante, on v\'erifie directement que l'application $\iota_\mu:\psi\mapsto \int k_\psi d\mu$, o\`u $k_\psi$ est la fonction \'etag\'ee associ\'ee \`a $\psi$, est un morphisme vers $\R$ envoyant $\vpi$ sur 1, et il en d\'ecoule facilement que $\iota_\mu=\mo$ (et en particulier que
$\iota_\mu$ prend des valeurs enti\`eres et ne d\'epend pas de $\mu$).

Pour prouver le th\'eor\`eme \ref{tab}, on a besoin du lemme classique suivant, qui servira aussi dans la preuve de la moyennabilit\'e; il semble remonter \`a Putnam \cite[Sec.~5]{Pu}; m\^eme si tous les ingr\'edients de la preuve y sont contenus, le r\'esultat n'y est pas explicit\'e; cependant les papiers ult\'erieurs s'y r\'ef\`erent. Une preuve directe est donn\'ee par Juschenko et Monod dans \cite{JM} (voir aussi Matui \cite[Proposition 3.2]{Ma1}, dans un langage un peu diff\'erent mais \`a port\'ee plus g\'en\'erale).

\begin{lem}[Putnam]\label{lputnam}
Si $X$ est compact totalement s\'epar\'e et $\vpi$ est un auto\-hom\'eo\-mor\-phisme minimal, alors pour tout $x\in X$ le stabilisateur $\lb\vpi\rb_{[x\rangle}$ de l'orbite positive $\vpi^\N(x)$ est localement fini.
\end{lem}
\begin{proof}
Soit $i_x$ la fonction orbitale $n\mapsto \vpi^n(x)$; soit $j_x$ l'action de $\lb\vpi\rb$ sur $\Z$ obtenue \`a partir de celle sur l'orbite de $x$ via la bijection $i_x:\Z\to\vpi^\Z(x)$; autrement dit, pour $\psi\in\lb\vpi\rb$, on d\'efinit $j_x(\psi)$ comme la permutation $i_x^{-1}\circ \psi\circ i_x$ de $\Z$.

L'assertion du lemme revient \`a affirmer que toute partie finie $F$ (qu'on peut supposer sym\'etrique) de $\lb\vpi\rb_{[x\rangle}$ engendre un groupe agissant sur $\Z$ (via $j_x$) avec orbites finies. Posons $\tau_x(\psi)(n)=j_x(\psi)(n)-n$, si bien que $\tau_x(\psi):\Z\to\Z$ est une fonction born\'ee. Fixons $F$. Soient $$m=\sup_{\theta\in F,\,n\in\Z}|\tau_x(\theta)(n)|,\quad A=\{-m,\dots,m\}^F.$$ On a une application anti-\'equivariante naturelle $\alpha$ continue de $X$ dans le d\'ecalage $(A^\Z,T)$ (avec $T(w)(n)=w(n-1)$), qui envoie $y\in X$ sur la suite $n\mapsto (\theta\mapsto \tau_y(\theta)(n))$. Son image est un sous-d\'ecalage minimal. Ceci implique facilement (voir par exemple \cite[Theorem 7.2]{MH}) que si on d\'efinit $I$ comme l'ensemble des $n\in\Z$ tels que $T^{-n}\alpha(x)$ et $\alpha(x)$ co\"\i ncident sur $\{0,\dots,m-1\}$, alors $I$ est coborn\'e, au sens o\`u la fonction $d(\cdot,I)$ est born\'ee sur $\Z$. On peut indexer $I$ de fa\c con strictement croissante: 
\mbox{$I=\{p_n,n\in\Z\}$,} si bien que la suite $(p_{n+1}-p_n)_{n\in\Z}$ est born\'ee. On d\'efinit l'intervalle entier \mbox{$P_n=\{p_n,p_n+1,\dots,p_{n+1}-1\}$,} de telle mani\`ere que $\Z$ est la r\'eunion disjointe des $P_n$. Si $0\le k\le m-1$ et $\theta\in F$, alors \begin{align*}j_x(\theta)(p_n+k)-p_n= & \tau_x(\theta)(p_n+k)+k\\
= & T^{-p_n}\tau_x(\theta)(k)+k=T^{-p_n}\alpha(x)(k)(\theta)+k\\
=& \alpha(x)(k)(\theta)+k=\tau_x(\theta)(k)+k=j_x(\theta)(k)\ge 0 
\end{align*}
car $\theta\in \lb\vpi\rb_{[x\rangle}$. Comme les \'el\'ements de $F$ translatent les points d'au plus $m$, cela d\'emontre que $j_x(\theta)(P_n)\subset\bigcup_{n'\ge n}P_{n'}$ pour tout $\theta\in F$, et comme $F$ est suppos\'e sym\'etrique, le m\^eme argument montre que $j_x(\theta)(P_n)\subset\bigcup_{n'\le n}P_{n'}$, et par cons\'equent  chaque $P_n$ est stable par $j_x(F)$. Comme ils sont de cardinal born\'e, la preuve est termin\'ee.
\end{proof}

\begin{proof}[Preuve du th\'eor\`eme \ref{tab}]
Soit $\psi\in \lb\vpi\rb^0$. Soit $m$ une borne sup\'erieure sur la fonction \'etag\'ee associ\'ee. 
Soient $I=\vpi^{N^c}(x)\cap\psi^{-1}(\vpi^{\N}(x))$ et $J=\vpi^N(x)\cap\psi^{-1}(\vpi^{\N^c}(x))$. Par hypoth\`ese, $I$ et $J$ ont m\^eme cardinal et sont inclus dans $\vpi^{\{-m,\dots,m\}}(x)$. Le fait \ref{fait_sn} montre qu'il existe une involution $s\in\lb\vpi\rb$ qui \'echange les \'el\'ements de $I$ et les \'el\'ements de $J$, et dont la fonction \'etag\'ee associ\'ee est born\'ee par $2m$. Dans la preuve du fait, en choisissant leurs supports assez petits, on peut supposer qu'ils \'evitent l'ensemble fini $\vpi^{\{-2m,\dots,2m\}}(x)$. Par cons\'equent, l'\'el\'ement $s\psi$ stabilise $\vpi^\N(x)$. Par le lemme \ref{lputnam}, $s\psi$ est d'ordre fini.
\end{proof}

On va maintenant donner une version plus pr\'ecise du th\'eor\`eme \ref{tab}, due \`a Matui.

\begin{propp}\label{02}
Le groupe $\lb\vpi\rb^0$ est engendr\'e par ses \'el\'ements d'ordre 2. 
\end{propp}

\begin{proof}
Dans la preuve du th\'eor\`eme \ref{tab}, on a montr\'e, apr\`es avoir fix\'e 
\mbox{$x\in X$} que tout \'el\'ement de $\lb\vpi\rb^0$ est produit d'une involution et d'un \'el\'ement de $\lb\vpi\rb_{[x\rangle}$, qui est localement fini par le lemme \ref{lputnam}. Il suffit donc de montrer que tout \'el\'ement d'ordre fini est produit d'\'el\'ements d'ordre 2. L'argument du lemme \ref{decompo} montre que tout \'el\'ement d'ordre fini est produit (\`a supports disjoints) d'\'el\'ements $\psi_i$ d'ordre fini $n_i$ d\'efinissant une action libre d'un groupe cyclique $\Z/n_i\Z$ sur le support de $\psi_i$. Par le lemme \ref{l:domaine}, il existe un clouvert $U$ tel que le support de $\psi_i$ est union disjointe des $(\vpi(U))_{0\le i\le n_i-1}$. L'argument du fait \ref{fait_sn} montre qu'on peut plonger le groupe sym\'etrique sur $\{0,\dots,n_i-1\}$ dans $\lb\vpi\rb$, de fa\c con \`a envoyer le $n_i$-cycle $(012\dots)$ sur $\psi_i$. Donc $\psi_i$ est produit d'\'el\'ements d'ordre~2.
\end{proof}

\begin{cor}
Le quotient $\lb\vpi\rb^0/\lb\vpi\rb'$ est un 2-groupe ab\'elien \'el\'ementaire et l'ab\'elianis\'e $\lb\vpi\rb/\lb\vpi\rb'$ est isomorphe au produit direct de $\Z$ et de $\lb\vpi\rb^0/\lb\vpi\rb'$.\qed
\end{cor}

La classe d'isomorphie du groupe $\lb\vpi\rb^0/\lb\vpi\rb'$ est donc d\'etermin\'ee par sa dimension comme espace vectoriel sur $\Z/2\Z$. Celle-ci d\'epend, en g\'en\'eral, de $\vpi$. Matui \cite{Ma1} prouve que ce groupe est isomorphe au groupe des co-invariants sous $\vpi$ du groupe des fonctions continues de $X$ dans $\Z/2\Z$.

\subsection{Approximation et pr\'esentation infinie}\label{approxi}

\begin{defn}\label{lef}
Un \mbox{(semi-)}\penalty0 groupe $G$ est {\em finiment approximable} (ou LEF {\og locally embeddable into finite groups\fg}) s'il v\'erifie l'une des conditions \'equivalentes suivantes
\begin{itemize}
\item tout syst\`eme fini d'\'egalit\'es et in\'egalit\'es sans param\`etre ayant une solution dans $G$ a une solution dans un \mbox{(semi-)}\penalty0 groupe fini;
\item $G$ est isomorphe \`a un sous-groupe d'un ultraproduit de \mbox{(semi-)}\penalty0 groupes finis;
\item (pour $G$ groupe d\'enombrable) $G$ est isomorphe au quotient d'un groupe r\'esiduellement fini par un sous-groupe distingu\'e $N$ qui est r\'eunion d'une suite croissante de sous-groupes distingu\'es finis de $G$.
\end{itemize}
\end{defn}

Un groupe est finiment approximable comme groupe si et seulement s'il l'est comme \mbox{(semi-)}\penalty0 groupe. Un \mbox{(semi-)}\penalty0 groupe est finiment approximable si et seulement si tout ses sous-\mbox{(semi-)}\penalty0 groupes de type fini le sont. Un \mbox{(semi-)}\penalty0 groupe r\'esi\-duel\-lement fini est finiment approximable; la r\'eciproque est vraie pour les \mbox{(semi-)}\penalty0 groupes de pr\'esentation finie mais pas pour les groupes de type fini (ces observations sont dues \`a St\"epin \cite{St2}). Il d\'ecoule de la d\'efinition qu'ils forment une classe stable par passage aux sous-\mbox{(semi-)}\penalty0 groupes et limites inductives filtrantes. On va montrer le r\'esultat suivant.

\begin{thm}[Grigorchuk-Medynets \cite{GrMe}]\label{clef}
Pour tout $X$ compact totalement s\'epar\'e muni d'une action minimale de $\Z$, le semi-groupe (et donc le groupe) plein-topologique est finiment approximable.
\end{thm}

Notons que ce r\'esultat est faux pour un auto\-hom\'eo\-mor\-phisme arbitraire, voir l'exemple de la proposition \ref{h3}.

\begin{cor}[Matui \cite{Ma1}]\label{non_pf}
Pour $\Gamma=\Z$, consid\'erons un sous-d\'ecalage $(\Gamma,X)$ minimal infini. Soit $\Lambda$ un sous-groupe de $\lb\Gamma\rb$ (ou sous-semi-groupe de $\lb\Gamma\rb\os$) contenant $\lb\Gamma\rb'$. Alors il n'est pas de pr\'esentation finie.
\end{cor}
\begin{proof}
Par le th\'eor\`eme \ref{clef}, $\Lambda$ est finiment approximable; si par l'absurde il est de pr\'esentation finie, il en d\'ecoule qu'il est aussi r\'esiduellement fini, et donc que $\lb\Gamma\rb'$ est aussi r\'esiduellement fini. Or c'est un groupe simple par le th\'eor\`eme \ref{dersim}, et il est infini par le fait \ref{fait_sn}, ce qui est contradictoire.
(Cet argument, bas\'e sur la r\'esiduelle finitude, combine l'id\'ee de Grigorchuk-Medynets \cite{GrMe} d'utiliser la r\'esiduelle finitude, avec, en amont, l'approche de Matui qui a permis de montrer le th\'eor\`eme \ref{clef}.)
\end{proof}

\begin{rema}
Il est fructueux et naturel de penser \`a un r\'esultat de pr\'esen\-tation infinie $G=G_\infty$ comme un r\'esultat d'approximation: un groupe (de type fini) de pr\'esentation infinie peut \^etre approch\'e par des groupes $G_n$ obtenus en tronquant une pr\'esentation sur un nombre fini de g\'en\'erateurs, ou parfois en prenant des quotients $G_n/H_n$ de ces derniers (par exemple finis) qui ne sont pas quotients de~$G_\infty$.
\end{rema}

Le th\'eor\`eme \ref{clef} se prouve par sp\'ecification au cas des sous-d\'ecalages, auquel cas il est valable sous des conditions moins restrictives que la minimalit\'e. On a besoin d'introduire quelques notions classiques.

\begin{defn}
Dans un d\'ecalage $A^\Z$, un {\em motif} est une fonction partiellement d\'efinie, \`a domaine de d\'efinition un intervalle entier fini, de $\Z$ vers $A$, \`a translation pr\`es. Un {\em motif} d'un sous-d\'ecalage $X\subset A^\Z$ est {\em autoris\'e} ou {\em interdit} selon qu'il est ou non la restriction d'un \'el\'ement de $X$ \`a un intervalle entier fini.

On dit qu'un sous-d\'ecalage $X\subset A^\Z$ est {\em de type fini} s'il est d\'efini par un nombre fini de motifs interdits, autrement dit s'il existe un clouvert $U\subset A^\Z$ tel que $X=\bigcap_{n\in\Z}T^nU$, o\`u $T$ est la fonction de d\'ecalage sur $A^\Z$.

On dit qu'un sous-d\'ecalage $X\subset A^\Z$ est {\em irr\'eductible} si pour tout couple $(w_1,w_2)$ de motifs autoris\'es dans $X$, il existe un motif autoris\'e de $X$ admettant $w_1$ comme segment initial et $w_2$ comme segment terminal.
\end{defn}

Notons qu'un sous-d\'ecalage $X\subset A^\Z$ est caract\'eris\'e par l'ensemble de ses motifs autoris\'es. On v\'erifie ais\'ement qu'un sous-d\'ecalage de type fini irr\'eductible sur $\Z$ a un ensemble dense de points p\'eriodiques, et qu'un sous-d\'ecalage minimal sur $\Z$ est \'egalement irr\'eductible (voir \cite{MH}).

\begin{propp}\label{sdlef}
Pour tout sous-d\'ecalage irr\'eductible sur $\Z$, le semi-groupe (et donc le groupe) plein-topologique est finiment approximable. 
\end{propp}

\begin{lem}\label{intirr}
Soit $X\subset A^\Z$ un sous-d\'ecalage de type fini irr\'eductible sur $\Z$ et non vide. Alors $X$ est l'intersection d\'ecroissante d'une suite de sous-d\'ecalages de type fini irr\'eductibles.
\end{lem}
\begin{proof}
Pour tout ensemble de mots $Z$, on note $X_Z$ le sous-d\'ecalage d\'efini par l'ensemble de mots interdits $Z$.
Pour $n\ge 1$, soit $P_n$ l'ensemble des motifs interdits de taille $n$. Clairement $X=\bigcap_{n\ge 1}X_{P_n}$ et chaque $X_{P_n}$ est un sous-d\'ecalage de type fini; il reste \`a v\'erifier que $X_{P_n}$ est irr\'eductible. Consid\'erons des motifs $w_1$ et $w_2$ de $X_{P_n}$, qu'on peut supposer de taille au moins $n$. Soient $u_1$ le segment terminal de taille $n$ de $w_1$ et $u_2$ le segment initial de taille $n$ de $w_2$. Comme $u_1$ et $u_2$ sont des motifs de $X$, il existe, par irr\'eductibilit\'e, un motif $v$ de $X$ ayant $u_1$ pour segment initial et $u_2$ pour segment terminal. En \'ecrivant $v=u_1uu_2$, on en d\'eduit que $w_1uw_2$ est autoris\'e dans $X_{P_n}$, ce qui d\'emontre que $X_{P_n}$ est irr\'eductible.
\end{proof}

\begin{proof}[Preuve de la proposition \ref{sdlef}]
Si $X$ est fini, il n'y a rien \`a d\'emontrer. On suppose donc $X$ infini. Il d\'ecoule alors de l'irr\'eductibilit\'e que $X$ est sans point isol\'e.

Gr\^ace au lemme \ref{intirr}, on \'ecrit $X=\bigcap X_n$ (intersection d\'ecroissante) o\`u $X_n\subset A^\Z$ est un sous-d\'ecalage ayant un ensemble dense de points p\'eriodiques. On affirme que le morphisme naturel $\underrightarrow{\lim}\lb X_n\rb\os\to \lb X\rb\os$ est un isomorphisme. La surjectivit\'e d\'ecoule de celle de $\lb X_n\rb\os\to \lb X\rb\os$ pour tout $n$; montrons l'injectivit\'e. Par la surjectivit\'e de $\lb A^\Z\rb\os\to\lb X\rb\os$, cela revient \`a v\'erifier que si $\psi,\psi'\in \lb A^\Z\rb\os$ et $\psi$ et $\psi'$ co\"\i ncident sur $X$, alors ils co\"\i ncident sur $X_n$ pour $n$ assez grand. Il suffit, par compacit\'e, de le montrer localement sur $x$. En supposant le contraire, il existe donc $x\in X$ tel que pour tout voisinage clouvert $V$ de $x$ et tout $n$, les fonctions $\psi$ et $\psi'$ ne co\"\i ncident pas sur $V\cap X_n$. On choisit $V$ sur qui $\psi$ et $\psi'$ sont des translations par des entiers distincts $k\neq k'$. Soit $\ell=|k-k'|$. Ainsi, tout \'el\'ement de $U\cap X$ est $\ell$-p\'eriodique, si bien que $U\cap X$ est fini, et donc $x$ est isol\'e dans $X$, ce qui est contradictoire.

Or $\lb X_n\rb\os$ agit sur $X_n$ avec un ensemble dense d'orbites finies, donc est r\'esi\-duellement fini. Comme \^etre finiment approximable est stable par limites inductives filtrantes, on en d\'eduit que $\underrightarrow{\lim}\lb X_n\rb\os$ l'est aussi, et donc $\lb X\rb\os$ est finiment approximable.
\end{proof}

\begin{rema}
Une preuve de la proposition \ref{sdlef} dans le cadre restreint aux groupes pleins-topologiques (au lieu des semi-groupes) aurait \'et\'e compliqu\'ee par le fait que les applications de restriction ne sont pas forc\'ement continues.
\end{rema}

\begin{proof}[D\'emonstration du th\'eor\`eme \ref{clef}] Montrons que le semi-groupe $\lb X\rb\os$ est 
finiment approximable. Il suffit de montrer que quelle que soit la partie finie $F$, il existe un morphisme $u$ du sous-semi-groupe engendr\'e par $F$ vers un groupe qu'on sait d\'ej\`a finiment approximable, et qui est injectif en restriction \`a $F$. On indexe injectivement $F=\{f_i\}$ et on choisit pour $i\neq j$ un \'el\'ement $x_{ij}\in X$ tel que $f_i(x_{ij})\neq f_j(x_{ij})$.
On consid\`ere une partition finie de $X$ en clouverts, sur lesquels les $f_i$  se trivialisent, et s\'eparant $f_i(x_{ij})$ et $f_j(x_{ij})$ pour tous $i\neq j$. On consid\`ere l'alg\`ebre bool\'eenne $\Z$-invariante engendr\'ee par cette partition, $Y$ le quotient correspondant. On voit que $Y$ est minimal et, par le fait \ref{fait_sd}, est un sous-d\'ecalage. On note
$\lb X\rt Y\rb\os$ l'ensemble des \'el\'ements de $\lb X\rb\os$ dont la fonction \'etag\'ee associ\'ee est constante sur les fibres de $X\to Y$ (ceci est d\'efini en toute g\'en\'eralit\'e ($\Gamma$ quelconque) par W.~Krieger \cite{Kri}). Ainsi on a un morphisme $u$ de semi-groupes canonique $\lb X\rt Y\rb\os\to\lb Y\rb\os$; on remarque que $F\subset \lb X\rt Y\rb\os$ et que par construction $f_i$ et $f_j$ ont des images distinctes par cette projection, pour tous $i\neq j$. Comme $\lb Y\rb\os$ est finiment approximable par la proposition \ref{sdlef}, on en d\'eduit que $\lb X\rb\os$ l'est aussi.
\end{proof}

\subsection{Sous-groupes ab\'eliens libres et sous-semi-groupes libres}\label{sgl}

Dans tout groupe, il est naturel de s'int\'eresser \`a la question d'existence de sous-groupes et sous-semi-groupes libres. On commence par la proposition suivante.

\begin{propp}
Soient $X$ un compact totalement s\'epar\'e infini et $\vpi$ un auto\-hom\'eo\-mor\-phisme minimal. Alors $\lb\vpi\rb'$ n'est pas de torsion; il contient en fait un groupe ab\'elien libre de rang infini.
\end{propp}
\begin{proof}On choisit un clouvert non vide $U\subset X$ disjoint de $V=\vpi(U)$. Soit $\sigma$ l'\'el\'ement d'ordre deux qui \'echange $U$ et $V$ par action de $\vpi^{\pm 1}$ et agit par l'identit\'e ailleurs. Alors $\vpi_V=\sigma\vpi_U\sigma^{-1}$, et on voit donc que $\sigma$ et $\vpi_U$ engendrent un groupe isomorphe au groupe $\Z^2\rtimes\langle\sigma\rangle$ (qui est un produit en couronne $\Z\wr(\Z/2\Z)$). Ce groupe contient bien la diagonale de $\Z^2$ dans son sous-groupe d\'eriv\'e.

Si $W$ est un clouvert de $X$, notons $\lb\vpi\rb_{[W]}$ l'ensemble des \'el\'ements de $\lb\vpi\rb$ qui sont l'identit\'e hors de $W$. Soit $W$ un clouvert non vide et montrons que $\lb\vpi\rb_{[W]}$ n'a pas un d\'eriv\'e de torsion. Comme l'action est minimale, la fonction de retour $\vpi_W$ est bien d\'efinie, et on a $\lb \vpi_W\rb=\lb\vpi\rb_{[W]}$. Le cas pr\'ec\'edent s'applique donc.

Consid\'erant une famille infinie $(W_n)$ de clouverts non vides deux \`a deux disjoints, les sous-groupes $\lb\vpi\rb'_{[W_n]}\subset \lb\vpi\rb'$ ont des supports deux \`a deux disjoints, si bien que $\lb\vpi\rb'$ contient une copie du produit direct restreint $\bigoplus_n\lb\vpi\rb'_{[W_n]}$. Comme $\lb\vpi\rb'_{[W_n]}$ contient un sous-groupe cyclique infini, on en d\'eduit que $\lb\vpi\rb'$ contient un groupe ab\'elien libre de rang infini d\'enombrable.
\end{proof}

Notons que le groupe de l'allumeur de r\'everb\`eres $(\Z/2\Z)\wr\Z$ contient un semi-groupe libre \`a deux g\'en\'erateurs (cela se voit directement, et a \'et\'e largement g\'en\'eralis\'e, par exemple par Rosenblatt \cite{Ros} au cas des groupes de type fini r\'esolubles non virtuellement nilpotents). Le th\'eor\`eme suivant, d\^u \`a Matui \cite{Ma2}, g\'en\'eralise une construction de Dahmani, Fujiwara et Guirardel \cite{DFG} dans le cas des \'echanges d'intervalles.

Rappelons qu'un groupe est {\em \`a croissance exponentielle} s'il poss\`ede une partie finie $S$ telle qu'en notant $S^n$ l'ensemble des produits de $n$ \'elements de $S$, on ait 
\mbox{$\underline{\lim}\#(S^n)^{1/n}>1$.}

\begin{thm}[Matui]\label{expo}
Soient $X$ un compact totalement s\'epar\'e et $\vpi$ un auto\-hom\'eo\-mor\-phisme minimal. On suppose en outre que $(X,\vpi)$ n'est pas un odom\`etre (cf. d\'efinition \ref{dodo}).
Alors tout sous-groupe du groupe plein-topologique $\lb\vpi\rb$ contenant $\lb\vpi\rb'$ contient un sous-groupe isomorphe au groupe de l'allumeur de r\'everb\`eres. En particulier, il contient un semi-groupe libre \`a deux g\'en\'erateurs et a une croissance exponentielle.
\end{thm}
\begin{proof}
Comme $X$ n'est pas un odom\`etre, commen\c cons par observer que, par compacit\'e, il existe $x\in X$ poss\'edant des voisinages clouverts arbitrairement petits \`a orbite infinie dans $\Clo(X)$. 

Il suffit bien entendu de d\'emontrer le th\'eor\`eme dans le cas de $\lb\vpi\rb'$, mais commen\c cons par $\lb\vpi\rb$ pour \'eviter, dans un premier temps, des complications techniques non essentielles. Comme $x$ n'est pas fixe, on peut se donner un voisinage clouvert $U$ de $x$ disjoint de $\vpi(U)$, et dont l'orbite dans $\Clo(X)$ par $\vpi$ est infini. Soit $\psi=\vpi_U$ la fonction de premier retour. On remarque que $\vpi_U$ n'est pas un odom\`etre: sinon, il existerait un entier $m\ge 1$, une partition $(U_i)_{i\in\Z/m\Z}$ de $U$ en clouverts et des $n_i\in\Z$ tels que $\psi(U_i)=U_{i+1}$ pour tout $i$ et $\psi=\vpi^{n_i}$ sur $U_i$. Si $n=\sum n_i$, on en d\'eduit que $\psi^m=\vpi^n$ sur $U$, donc $\vpi^n(U)=U$ contredisant le choix de $U$. Il existe donc un clouvert $V$ inclus dans $U$, dont la $\psi$-orbite est infinie.

V\'erifions que dans le $\Z/2\Z$-espace vectoriel $(\Z/2\Z)^U$, la famille des $\mathbf{1}_{\psi^nV}$ est libre: sinon on a une relation du type $\sum_{k=a}^b\eps_i\mathbf{1}_{\psi^nV}=0$ avec $a\le b$, $\eps_i\in\{0,1\}$ et $\eps_a=\eps_b=1$. Elle montre, en translatant, que pour tout $k\ge b$, l'\'el\'ement $\mathbf{1}_{\psi^kV}$ est dans le sous-groupe additif engendr\'e par les $\mathbf{1}_{\psi^\ell V}$ pour $\ell\in\{a,\dots k-1\}$; par r\'ecurrence il en d\'ecoule qu'il appartient en fait au sous-groupe additif engendr\'e par les $\mathbf{1}_{\psi^\ell V}$ pour $\ell\in\{a,\dots b-1\}$, qui est un sous-groupe fini ind\'ependant de $k$. 
Ceci contredit que les $(\mathbf{1}_{\psi^kV})_{k\ge b}$ sont deux \`a deux 
distincts; la famille $(\mathbf{1}_{\psi^nV})_{n\in\Z}$ est donc libre.

Si $F\in (\Z/2\Z)^{(\Z)}$ est une partie finie de $\Z$, on d\'efinit $A_F=\sum_{n\in F}\mathbf{1}_{\psi^nU}\in(\Z/2\Z)^U$, et $\sigma_F\in\lb\vpi\rb$ comme l'involution qui \'echange $A_F$ et $\vpi(A_{F})$ par action de $\vpi^{\pm 1}$, et agit par l'identit\'e ailleurs. On a $\sigma_F\in\lb\vpi\rb$ et l'application $F\mapsto\sigma_F$ est un morphisme injectif, par ce qui pr\'ec\`ede.
 Si on d\'efinit $\Psi$ comme le produit (commutatif) de $\psi$ et $\vpi\circ\psi\circ\vpi^{-1}$, alors on v\'erifie que $\Psi\circ\sigma_F\circ\Psi^{-1}=\sigma_{T(F)}$, o\`u $T:n\mapsto n+1$ est la translation de $\Z$. Ainsi $\Psi$ et $\sigma_{\{0\}}$ engendrent un groupe $H$ isomorphe au groupe de l'allumeur de r\'everb\`eres. 

Pour obtenir un allumeur \`a l'int\'erieur de $\lb\vpi\rb'$, on suppose en outre que les $\vpi^iU$ sont disjoints pour $0\le i\le 3$ (on peut le supposer car l'orbite de $x$ a au moins 4 \'el\'ements). Consid\'erons le groupe $H$ pr\'ec\'edemment construit, et l'\'el\'ement d'ordre deux $s$ \'echangeant $U\cup \vpi(U)$ et $\vpi^2(U)\cup \vpi^3(U)$ par action de $\vpi^{\pm 2}$ et agissant par l'identit\'e ailleurs; on voit que le groupe engendr\'e par $s$ et $H$ est isomorphe au produit semidirect $(H\times H)\rtimes \langle s\rangle$, o\`u $s$ agit par permutation des facteurs. Il est facile de voir que ce groupe contient, dans son d\'eriv\'e, un groupe isomorphe \`a l'allumeur, par exemple le sous-groupe de $H\times H$ engendr\'e par $(\Psi,\Psi^{-1})$ et $(\sigma_{\{0\}},\sigma_{\{0\}})$.
\end{proof}

\begin{rema}
R\'eciproquement, dans le cas d'un odom\`etre $X$ (pour une action d'un groupe ab\'elien $\Gamma$ quelconque), on v\'erifie sans peine que tout sous-groupe de type fini de $\lb\Gamma,X\rb$ est virtuellement ab\'elien, donc \`a croissance polynomiale; en particulier, $\lb\Gamma,X\rb$ ne contient pas de sous-semi-groupe libre \`a deux g\'en\'erateurs.
\end{rema}

\subsection{Autres r\'esultats}
Mentionnons quelques autres r\'esultats sur la structure de $\lb\vpi\rb$, quand $\vpi$ est un auto\-hom\'eo\-mor\-phisme minimal d'un espace compact totalement s\'epar\'e.

\subsubsection{Isomorphismes et automorphismes}\label{isoa}

Un r\'esultat de Giordano, Putnam et Skau \cite{GPS}, sous une version l\'eg\`erement am\'elior\'ee par Bezuglyi et Medynets \cite{BeM} (mais avec une approche diff\'erente) montre que si $\vpi_1,\vpi_2$ sont des auto\-hom\'eo\-mor\-phismes minimaux d'espaces de Cantor $X_i$ et, si pour $i=1,2$ on se donne un sous-groupe $\Lambda_i$  de $\lb\vpi_i\rb$ contenant $\lb\vpi_i\rb'$, alors tout isomorphisme $\Lambda_1\to\Lambda_2$ est induit par un hom\'eo\-mor\-phisme $h:X_1\to X_2$, au sens o\`u $\Lambda_2=h\Lambda_1h^{-1}$. En particulier, le groupe des automorphismes de $\Lambda_1$ s'identifie au normalisateur de $\Lambda_1$ dans $\Homeo(X_1)$. 

D'autre part, un th\'eor\`eme de la th\`ese de Boyle \cite{Bo} (disponible dans \cite{GPS95} ou, g\'en\'eralis\'e, dans \cite{BoTo}) montre que pour un auto\-hom\'eo\-mor\-phisme minimal $\vpi$ de l'espace de Cantor $X$, si $\psi\in\lb\vpi\rb$ a les m\^emes orbites que $\vpi$ alors le groupe cyclique $\langle\psi\rangle$ est conjugu\'e \`a $\langle\vpi\rangle$ dans $\lb\vpi\rb$. Giordano, Putnam et Skau combinent ces r\'esultats pour montrer que, sous les hypoth\`eses pr\'ec\'edentes, $\langle\vpi_1\rangle$ et $\langle\vpi_2\rangle$ sont conjugu\'es dans $\Homeo(X)$, et d\'eduisent \'egalement que le groupe des automorphismes de tout sous-groupe $\Lambda$ de $\lb\vpi\rb$ contenant $\lb\vpi\rb'$ est engendr\'e par $\lb\vpi\rb$ et par le commutant de $\vpi$ dans $\Homeo(X)$. En particulier, le groupe des automorphismes ext\'erieurs de $\lb\vpi\rb^0$ (voir \S\ref{abelianise}) s'identifie au commutant $C_\vpi$ de $\vpi$ dans $\Homeo(X)$, et celui de $\lb\vpi\rb$, au quotient de $C_\vpi$ par son sous-groupe cyclique central~$\langle\vpi\rangle$. 

\subsubsection{Distorsion}

Dans $\lb\vpi\rb$, pour $\vpi$ auto\-hom\'eo\-mor\-phisme minimal d'un espace compact totalement s\'epar\'e infini $X$, les sous-groupes cycliques sont non distordus; autrement dit pour toute partie finie $S$ et tout sous-groupe cyclique infini $\langle z\rangle$ inclus dans le sous-groupe engendr\'e par $S$, la longueur de $z$ par rapport \`a $S$ cro\^\i t lin\'eairement. En effet, en fixant $x\in X$ et en consid\'erant l'action $j_x$ de $\lb\vpi\rb$ sur $\Z$ par permutations \`a d\'eplacement born\'e obtenue en identifiant l'orbite de $x$ \`a $\Z$ (voir la preuve du lemme \ref{lputnam}), on voit que si la longueur de $(z^n)$ cro\^\i t sous-lin\'eairement, alors ses orbites dans $\Z$ sont finies; en particulier, en prenant la r\'eunion des orbites contenues dans $\N$, on obtient une partie finie $A$ telle que $A\bigtriangleup\N$ est invariant par $z$. Or le lemme de Putnam (lemme \ref{lputnam}) implique que le stabilisateur de $\N+k$ est localement fini pour tout $k\in\Z$, et l'argument du d\'ebut de la preuve du crit\`ere \ref{cr_jm} implique alors que le stabilisateur de $A\bigtriangleup\N$ est localement fini, si bien que $z$ est de torsion.

Ce r\'esultat de non-distorsion des groupes cycliques est \'egalement vrai dans le cas du groupe des \'echanges d'intervalles (C.~Novak \cite{Nov}).

\subsection{Quelques exemples non minimaux}\label{enm}
On va indiquer quelques exemples de sous-d\'ecalages non minimaux sur $\Z$, pour qui la structure du groupe plein-topologique est tr\`es diff\'erente du cas d'un sous-d\'ecalage minimal.

On commence par une construction due \`a van Douwen \cite{VD}. 
Consid\'erons un alphabet $A=\{a_1,a_2,\dots,a_q\}$ \`a $q$ lettres. Consid\'erons le d\'ecalage propre $\vpi$ associ\'e \`a l'ensemble des mots bi-infinis propres \[A^\Z_{\textnormal{pro}}=\{w\in A^\Z\mid \forall n\in\Z,\; w(n)\neq w(n+1)\}.\]

\begin{propp}\label{psgl}
Si $q\ge 3$, alors le groupe plein-topologique $\lb A^\Z_{\textnormal{pro}}\rb$ contient un sous-groupe libre non ab\'elien (agissant fid\`element sur au moins une orbite).
\end{propp}
\begin{proof}
La preuve consiste \`a plonger dans $\lb A^\Z_{\textnormal{pro}}\rb$ le produit libre de $q$ groupes cycliques d'ordre 2, sachant que le produit libre de 3 groupes cycliques d'ordre 2 contient un sous-groupe d'indice 2 libre \`a 2 g\'en\'erateurs. Pour toute lettre $a_i$, on d\'efinit une involution $\sigma_i\in\lb A^\Z_{\textnormal{pro}}\rb$ comme ceci: si $w\in A^\Z_{\textnormal{pro}}$, on pose 
$$\sigma_i(w)=\vpi^{\kappa_i(w)}w,\textnormal{ o\`u}\quad\kappa_i(w) =
\left\{ \begin{array}{cl} 1         &        \mbox{si } w(0)=a_i; \\
  -1                 &               \mbox{si } w(1)=a_i; \\
   0                 &             \mbox{dans les autres cas},
\end{array}\right.$$
en remarquant que les conditions $w(0)=a_i$ et $w(1)=a_i$ sont exclusives. 
On v\'erifie bien que $\sigma_i$ est une involution \'echangeant les clouverts disjoints $\{\kappa_i=1\}$ et $\{\kappa_i=-1\}$ et appartient \`a $\lb A^\Z_{\textnormal{pro}}\rb$. V\'erifions que les $\sigma_i$ engendrent un produit libre de groupes cycliques d'ordre 2. Il faut montrer que pour tout mot non trivial $m=\prod_{j=1}^n\sigma_{{k_j}}$ sans deux lettres cons\'ecutives \'egales ($n\ge 1$ et $k_j\neq k_{j+1}$ pour tout $1\le j<n$), on a $m\neq 1$. En effet, on choisit $w$ tel que $w(j)=k_j$ pour tout $j=1,\dots,n$, tel que $w(0)\notin\{k_1,k_n\}$ et $w(n+1)\neq k_n$. Alors on v\'erifie que $m^{-1}\cdot w=\vpi^{-n}w\neq w$.

Remarquons qu'il existe dans $A^\Z_{\textnormal{pro}}$ une orbite dense (choisir un mot bi-infini propre contenant tout mot fini propre comme sous-mot) et $\lb A^\Z_{\textnormal{pro}}\rb$ agit donc fid\`element sur cette orbite.
\end{proof}

Cette construction est apparent\'ee \`a la premi\`ere preuve connue, due \`a Schreier \cite[p.~170]{Sch}, de la r\'esiduelle finitude des groupes libres.

\begin{cor}[Schreier]
Tout groupe libre est r\'esiduellement fini.
\end{cor}
\begin{proof}
Clairement on peut se ramener au cas d'un groupe libre de type fini; un tel groupe peut se plonger dans $\lb A^\Z_{\textnormal{pro}}\rb$ par la proposition.
Il  suffit alors de remarquer que comme la r\'eunion des orbites finies, c'est-\`a-dire l'ensemble des suites p\'eriodiques, est dense dans $A^\Z_{\textnormal{pro}}$, le groupe plein-topologique $\lb A^\Z_{\textnormal{pro}}\rb$ est r\'esiduellement fini.
\end{proof}

La construction de la proposition a \'et\'e adapt\'ee par Elek et Monod \cite{EM} pour exhiber un sous-d\'ecalage minimal sur $\Z^2$ dont le groupe plein-topologique contient un sous-groupe libre. 

Passons \`a des exemples plus \'el\'ementaires, toujours avec $\Gamma=\Z$. On consid\`ere l'alphabet $\{a,b\}$ et $Y\subset A^\Z$ l'ensemble des suites ne contenant pas le motif $ba$. Ainsi $Y$ est constitu\'e des deux suites constantes et des suites commen\c cant par une infinit\'e de $a$ et terminant par une infinit\'e de $b$. (C'est un sous-d\'ecalage de type fini mais pas irr\'eductible.) On peut identifier $Y$ \`a la compactification $\Z\cup\{\pm\infty\}$ o\`u $n\in\Z$ correspond \`a l'\'el\'ement pour lequel la transition de $a$ \`a $b$ est en position $(n,n+1)$, et $+\infty$ et $-\infty$ correspondent \`a la suite constante $a$ et $b$ respectivement. 

\begin{propp}\label{h2}
Sous cette identification, $\lb Y\rb$ est le groupe des permutations de $\Z$ qui co\"\i ncident avec une translation hors d'un ensemble fini.
\end{propp}

Ce groupe est bien connu. Il a un morphisme vers $\Z$, donn\'e par la longueur de translation \`a l'infini, dont le noyau est constitu\'e des permutations \`a support fini de $\Z$. Il est finiment approximable (voir la d\'efinition \ref{lef}) \cite[p.~127]{Neu}, mais pas r\'esiduellement fini puisqu'il contient un groupe simple infini. Il peut s'interpr\'eter comme le groupe des permutations du graphe de Cayley standard de $\Z$ qui pr\'eserve presque toute ar\^ete et fixe les bouts.

\begin{proof}[Preuve de la proposition \ref{h2}]
Soit $H_2$ le groupe d\'efini dans la proposition. V\'eri\-fions d'abord que $H_2\subset\lb Y\rb$. Comme $H_2$ est engendr\'e par la translation $\tau(n)=n+1$ et par la transposition $(0,1)$, il suffit de v\'erifier que cette derni\`ere est dans $\lb Y\rb$: c'est bien le cas car elle est d\'efinie par la fonction $\kappa$ d\'efinie par $\kappa(w)=1$ si $(w(0),w(1),w(2))=(a,b,b)$, par $\kappa(w)=-1$ si $(w(0),w(1),w(2))=(a,a,b)$ et $\kappa(w)=0$ sinon.

Montrons maintenant que tout $\sigma\in\lb Y\rb$ appartient \`a $H_2$. Par d\'efinition de $\lb Y\rb$, il existe $k$ tel que la fonction \'etag\'ee $\kappa$ d\'efinissant $\sigma$ est de la forme $\kappa(w)=\kappa'(w|_{\{-k,\dots,k\}})$. Soient $f_a,f_b$ les fonctions constantes \'egales \`a $a$ et $b$ sur $\{-k,\dots,k\}$. Soit $N_-=\kappa'(f_b)$ et $N_+=\kappa'(f_a)$. Alors on voit que $\sigma(n)=n+N_+$ pour tout $n\ge k$ et $\sigma(n)=n+N_-$ pour tout $n\le -k-1$. Cet argument montre en fait que tout $\sigma\in\lb Y\rb\os$ est une translation de chaque c\^ot\'e \`a partir d'un certain rang. Le fait que $\sigma$ est une permutation implique de plus que $N_+=N_-$, donc $\lb Y\rb\subset H_2$.
\end{proof}

L'exemple suivant est similaire. On consid\`ere un alphabet $A=\{a,b,c\}$ et on consid\`ere le sous-d\'ecalage $Y'$ d\'efini par l'interdiction de tous les motifs de deux lettres, sauf $aa$, $ab$, $bc$, $cb$. Ainsi $Y'$ est constitu\'e de l'ensemble $Z$ des suites translat\'ees de $\dots aaabcbcbc\dots$, et de trois points, la suite constante $a$ (not\'ee $+\infty$, et les deux suites $\dots bcbcbc\dots$, 
not\'ees $-\infty_{\textnormal{pair}}$ 
(pour celle avec $w(0)=c$) et $-\infty_{\textnormal{impair}}$ (pour celle avec $w(0)=b$). Comme $Z$ est dense, pour d\'ecrire $\lb Y'\rb$, il suffit de comprendre son action sur $Z$. On identifie $Z$ \`a $\Z$ en identifiant $n$ \`a l'unique suite $w\in Z$ telle $(w(n),w(n+1))=(a,b)$. Remarquons, pour la consistance des notations, que l'\'el\'ement correspondant \`a $2n$ (resp.\ $2n+1$) tend vers
$-\infty_{\textnormal{pair}}$ (resp.\ $-\infty_{\textnormal{impair}}$) quand $n$ tend vers $-\infty$.

\begin{propp}\label{h3}
Sous cette identification, $\lb Y'\rb$ est le groupe des permutations de $\Z$ qui, en restriction \`a chacune des parties $\N$, $-2\N$ et $-2\N-1$, co\"\i ncident avec une translation hors d'un ensemble fini. 
\end{propp}
La preuve, similaire \`a celle de la proposition \ref{h2}, est laiss\'ee au lecteur.

Ce groupe est connu, du moins son sous-groupe d'indice 2 constitu\'e des permutations qui, modulo un ensemble fini, pr\'eservent $-2\N$ et $-2\N-1$: c'est le groupe de Houghton $H_3$ (introduit dans \cite{Hou}). En effet, si on munit $\Z$ d'une structure de graphe en joignant $n$ \`a $n+1$ pour tout $n\ge -1$ et $n$ \`a $n+2$ pour tout $n\le -2$, le graphe obtenu est constitu\'e de 3 branches infinies jointes en 0; $H_3$ est par d\'efinition le groupe des permutations de $\Z$ qui pr\'eservent la structure de graphe sauf en un nombre fini d'ar\^etes et fixent les bouts. Ici, $\lb Y'\rb$ est son sur-groupe d'indice 2, le stabilisateur du bout $+\infty$. On d\'eduit que $\lb Y'\rb$ est un groupe de pr\'esentation finie (puisque $H_3$ l'est, par un r\'esultat non publi\'e de Burns et Solitar, prouv\'e par Brown dans \cite{Bro}). Comme il contient le groupe simple infini des permutations \`a support fini altern\'ees, $\lb Y'\rb$ n'est donc pas finiment approximable (voir la d\'efinition \ref{lef}), en contraste avec $\lb Y\rb$.

En g\'en\'eral, pour $n\ge 2$, le groupe de Houghton $H_n$ est le groupe des permutations d'un graphe constitu\'e de $n$ branches infinies jointes en un point, qui pr\'eservent la structure de graphe sauf en un nombre fini d'ar\^etes, et fixent les bouts.
Une g\'en\'eralisation facile de l'exemple pr\'ec\'edent fournit un sous-d\'ecalage de type fini dont le groupe plein-topologique contient le groupe de Houghton $H_n$ comme sous-groupe d'indice fini.

\section{Moyennabilit\'e}\label{smoy}
\subsection{Le crit\`ere}\label{lecri}
On va commencer par une construction g\'en\'erale. \'etant donn\'e un groupe $\Gamma$, on d\'efinit le groupe des {\em translations par morceaux} $\TM(\Gamma)$ comme le groupe plein-topologique associ\'e \`a l'action \`a gauche de $\Gamma$ sur l'espace topologique discret $\Gamma$. Si $\Gamma$ est muni d'une m\'etrique invariante \`a gauche dont les boules sont finies (tout groupe d\'enombrable poss\`ede une telle m\'etrique), on peut interpr\'eter ce groupe de mani\`ere purement m\'etrique, \`a savoir comme l'ensemble des permutations $\sigma$ de $\Gamma$ \`a {\em d\'eplacement born\'e} au sens o\`u $\sup_{\gamma\in\Gamma}d(\gamma,\sigma(\gamma))<\infty$.

Si $\Gamma$ agit par auto\-hom\'eo\-mor\-phismes sur un espace topologique $X$, et si $x\in X$, soit $i_x$ l'application orbitale $\gamma\mapsto \gamma x$ de $\Gamma$ vers $X$. Si $x$ a un stabilisateur trivial, $i_x$ est une bijection de $\Gamma$ vers l'orbite $\Gamma x$, et l'action de $\lb \Gamma,X\rb$, restreinte \`a l'orbite $\Gamma x$, est conjugu\'ee, via la bijection $i_x$, \`a une action de $\lb \Gamma,X\rb$ sur $\Gamma$ par translations par morceaux, donn\'ee par le morphisme $j_x:\lb \Gamma,X\rb\to\TM(\Gamma)$ d\'efini par $j_x(\gamma)=i_x^{-1}\circ \gamma\circ i_x$.

On a vu (proposition \ref{psgl}) que pour $\Gamma=\Z$, pour un sous-d\'ecalage convenable $(X,\vpi)$ avec une orbite dense, le groupe $\lb \Z,X\rb=\lb\vpi\rb$ peut contenir un sous-groupe libre agissant fid\`element sur l'orbite d'un point $x$. Par cons\'equent, en prenant l'image par $j_x$, on voit que le groupe $\TM(\Z)$ contient \'egalement des sous-groupes libres, et n'est donc pas moyennable. Cependant c'est en utilisant ce groupe que Juschenko et Monod prouvent leur r\'esultat.

\begin{thm}[Crit\`ere de Juschenko-Monod]\label{cr_jm}
Soit $G$ un groupe et consid\'erons une action de $G$ sur $\Z$ par permutations \`a d\'eplacement born\'e, c'est-\`a-dire, donn\'ee par un morphisme $G\to\TM(\Z)$. On suppose que le stabilisateur $\{g\in G\mid g(\N)=\N\}$ de $\N$ est moyennable. Alors $G$ est lui-m\^eme moyennable. 
\end{thm}

On discutera la preuve de ce crit\`ere plus loin. Commen\c cons par la principale cons\'equence.

\begin{cor}
Le groupe plein-topologique $\lb\vpi\rb$ associ\'e \`a un auto\-hom\'eo\-mor\-phisme minimal d'un compact totalement s\'epar\'e $X$ est moyennable.
\end{cor}

Le th\'eor\`eme s'applique: fixons en effet $x\in X$ et faisons agir $\lb\vpi\rb$ sur $\Z$ par permutations \`a d\'eplacement born\'e via $j_x$. Alors le stabilisateur de $\N$ est \'egal au stabilisateur $\lb\vpi\rb_{[x\rangle}$ de l'orbite positive $\{\vpi^{n}x:n\in\N\}$ de $x$, qui, par le lemme de Putnam (lemme \ref{lputnam}) est localement fini (c'est-\`a-dire r\'eunion croissante filtrante de groupes finis), donc moyennable. 

\subsection{Sur la d\'emonstration du crit\`ere}

Le crit\`ere est bas\'e sur le r\'esultat analytique suivant. Notons $\mathsf{C}_2$ le groupe cyclique d'ordre 2 et $\mathsf{C}_2^{(\Z)}$ le groupe des fonctions \`a support fini $\Z\to \mathsf{C}_2$, qui s'identifie \`a l'ensemble des parties finies de $\Z$.

\begin{thm}[Juschenko-Monod]\label{jm_an}
L'action affine de $\mathsf{C}_2^{(\Z)}\rtimes\TM(\Z)$ sur $\mathsf{C}_2^{(\Z)}$ pr\'eserve une moyenne.
\end{thm}
Ici, l'action est simplement l'action affine engendr\'ee par l'action $\Z/2\Z$-lin\'eaire de $\TM(\Z)$ et par le groupe des translations. Le th\'eor\`eme s'applique pour d\'emon\-trer le crit\`ere \ref{cr_jm} \`a l'aide du fait suivant.

\begin{fact}\label{stabm}
Si un groupe discret $G$ agit sur un ensemble $X$ en pr\'eservant une moyenne, et si chaque point a un stabilisateur moyennable, alors $G$ est lui-m\^eme moyennable.
\end{fact}
\begin{proof}
On commence par rappeler que $G$ est moyennable si, et seulement si, pour toute action de $G$ par transformations affines continues sur un espace localement convexe pr\'eservant un convexe compact non vide $K$, l'action fixe un point de $K$ (au besoin, on choisit cela comme d\'efinition de moyennabilit\'e; voir par exemple \cite{Gre}). Soit donc une telle action; il suffit en fait de montrer que $G$ pr\'eserve une probabilit\'e  sur les bor\'eliens de $K$, car le barycentre de cette derni\`ere (voir par exemple \cite[Theorem 2.29]{Luk}) est alors un point de $K$ fix\'e par~$G$.

On choisit donc un sous-ensemble $I\subset X$ contenant exactement un \'el\'ement de chaque $G$-orbite, et pour tout $i\in I$ soit $H_i$ son stabilisateur. Par moyennabilit\'e, $H_i$ fixe un point $\kappa_i\in K$. L'application $i\mapsto\kappa_i$ s'\'etend de mani\`ere unique en une fonction $G$-\'equivariante $X\to K$. L'image directe d'une moyenne invariante sur $X$ est une moyenne invariante sur $K$, d\'efinissant une forme lin\'eaire positive normalis\'ee $G$-invariante sur $\ell^\infty(K)$. Celle-ci se restreint en une forme lin\'eaire positive normalis\'ee sur $C(K)$, qui d\'efinit par le th\'eor\`eme de repr\'esentation de Riesz une unique mesure de probabilit\'e de Radon sur $K$, qui par unicit\'e est $G$-invariante. 
\end{proof}

\begin{proof}[Preuve du th\'eor\`eme \ref{cr_jm}]
Soit maintenant $G$ v\'erifiant le crit\`ere du th\'eor\`eme \ref{cr_jm} et montrons que $G$ est moyennable.

L'hypoth\`ese de stabilisateur implique en particulier que le noyau de l'action de $G$ sur $\Z$ est moyennable; la moyennabilit\'e \'etant stable par extensions et passage au quotient, on peut donc se ramener au cas d'une action fid\`ele, ce qui revient \`a supposer que $G$ est contenu dans $\TM(\Z)$.


V\'erifions d'abord que pour toute partie $B$ de $\Z$ telle que $\N\bigtriangleup B$ est finie, le stabilisateur $G_B$ de $B$ dans $G$ est moyennable. Commen\c cons par le cas 
de $B=\N+k=\{k,k+1,\dots\}$ pour $k\in\Z$. En effet, remarquons d'abord que tout sous-groupe de $\TM(\Z)$ pr\'eserve une moyenne sur chacune de ses orbites dans $\Z$. Si par l'absurde $G_{\N+k}$ est non moyennable, alors pour tout $j\in\Z$, en appliquant le fait \ref{stabm} \`a l'orbite de $j$ pour l'action de $G_{\N+k}$, on obtient que le stabilisateur $(G_{\N+k})_j$ est non moyennable. Or pour $j=k$ ce stabilisateur est contenu dans $G_{\N+k+1}$ et pour $j=k-1$ ce stabilisateur est contenu dans $G_{\N+k-1}$. Ceci montre que $G_{\N+k+1}$ et $G_{\N+k-1}$ sont non moyennables, et donc par une r\'ecurrence de chaque c\^ot\'e, $G_{\N+j}$ est non moyennable pour tout $j\in\Z$, contredisant pour $j=0$ l'hypoth\`ese du crit\`ere. Si maintenant $B$ est quelconque, il existe une permutation \`a support fini de $\Z$ \'echangeant $B$ et $\N+k$ pour un certain (unique) $k\in\Z$, on remarque alors que le stabilisateur $G_B$ est inclus dans le sous-groupe engendr\'e par le stabilisateur de $G_{\N+k}$, qui est moyennable par ce qui pr\'ec\`ede, et par le groupe des permutations \`a support fini, qui est distingu\'e et moyennable. Le stabilisateur $G_B$ est donc lui-m\^eme moyennable.
  
Terminons maintenant la preuve. On a une premi\`ere action, dite ordinaire, de $\TM(\Z)$ sur le $\Z/2\Z$-espace vectoriel $\mathsf{C}_2^{(\Z)}$, donn\'ee par $w\cdot \mathbf{1}_A=\mathbf{1}_{w(A)}$. On a \'egalement une action affine de $\TM(\Z)$ sur $\mathsf{C}_2^{(\Z)}$, de partie lin\'eaire l'action ordinaire, donn\'ee par $$w\cdot \mathbf{1}_A=\mathbf{1}_{w(A)}+\mathbf{1}_\N-\mathbf{1}_{w(\N)}.$$ Le stabilisateur de $\mathbf{1}_A$ pour cette action affine est \'egal au stabilisateur pour l'action ordinaire de la diff\'erence sym\'etrique $A\bigtriangleup\N$. Cette action affine pr\'eserve une moyenne par le th\'eor\`eme \ref{jm_an}.
Pour toute partie finie $A$, le stabilisateur de $\mathbf{1}_A$  dans $G$ pour l'action affine est \'egal au stabilisateur $G_{\N\bigtriangleup A}$ de $\N\bigtriangleup A$ pour l'action ordinaire dont on vient de v\'erifier la moyennabilit\'e. Le fait \ref{stabm} s'applique donc \`a l'action affine de $G$ sur $\mathsf{C}_2^{(\Z)}$, si bien que $G$ est moyennable.
\end{proof}

\subsection{Preuve du th\'eor\`eme \ref{jm_an}}

Le groupe $\mathsf{C}_2^\Z$ des fonctions $\Z\to\mathsf{C}_2$ est muni de la topologie produit, qui en fait un groupe compact, et de la mesure produit de la mesure \'equir\'epartie sur $\mathsf{C}_2=\{0,1\}$, qui est sa mesure de Haar.

Juschenko et Monod observent que l'existence d'une moyenne invariante comme dans le th\'eor\`eme \ref{jm_an} d\'ecoule de l'existence d'une suite d'\'el\'ements non nuls $f_n\in L^2(\mathsf{C}_2^\Z,\R)$ telle que
\begin{enumerate}[(1)]
\item\label{fnai} $(f_n)$ est asymptotiquement $\TM(\Z)$-invariante, c'est-\`a-dire que pour tout $g\in\TM(\Z)$ on a $\lim_{n\to\infty}\|f_n-gf_n\|_2/\|f_n\|_2=0$;
\item\label{fnze} $(f_n)$ est asymptotiquement support\'ee par l'hyperplan $H_0=\mathsf{C}_2^{\Z\smallsetminus\{0\}}=\{u\in \mathsf{C}_2^\Z:u(0)=0\}$, au sens o\`u $\lim_{n\to\infty}\|f_n\mathbf{1}_{H_0}\|_2/\|f_n\|_2=1$.
\end{enumerate}

En effet, consid\'erons une telle suite $(f_n)$. La suite des transform\'ees de Fourier $(\widehat{f_n})$ est une suite dans $\ell^2\big(\mathsf{C}_2^{(\Z)}\big)$, asymptotiquement $\TM(\Z)$-invariante par (\ref{fnai}). La transform\'ee de Fourier entrelace la multiplication par $\mathbf{1}_{H_0}$ \`a la projection sur l'ensemble des fonctions dans $\ell^2\big(\mathsf{C}_2^{(\Z)}\big)$ 
qui sont invariantes par $\mathsf{C}_2^{(\{0\})}$. Donc par (\ref{fnze}), la suite $(\widehat{f_n})$ est asymptotiquement $\mathsf{C}_2^{(\{0\})}$-invariante. Puisque $\mathsf{C}_2^{(\Z)}\rtimes\TM(\Z)$ est engendr\'e par $\mathsf{C}_2^{(\{0\})}$ et $\TM(\Z)$, on en d\'eduit que la suite $(\widehat{f_n})$ est asymptotiquement $\mathsf{C}_2^{(\Z)}\rtimes\TM(\Z)$-invariante; autrement dit en posant $u_n=|\widehat{f_n}|/\|f_n\|_2$, on a $\|u_n\|_2=1$ et $\|u_n-gu_n\|_2\to 0$ pour tout $g\in \mathsf{C}_2^{(\Z)}\rtimes\TM(\Z)$. Ainsi $\|u_n^2\|_1=1$ et, par l'in\'egalit\'e de Cauchy-Schwarz, on obtient
$$\|u_n^2-gu_n^2\|_1\le \|u_n+gu_n\|_2\|u_n-gu_n\|_2\le 2\|u_n-gu_n\|_2,$$
si bien que $\|u_n^2-gu_n^2\|_1\to 0$. 
Par cons\'equent, si $m$ est une valeur d'adh\'erence faible-* dans $\ell^\infty\big(\mathsf{C}_2^{(\Z)}\big)$ de la suite de probabilit\'es $(u_n^2)$ de $\ell^1\big(\mathsf{C}_2^{(\Z)}\big)$, alors $m$ d\'efinit une moyenne $\mathsf{C}_2^{(\Z)}\rtimes\TM(\Z)$-invariante sur~$\mathsf{C}_2^{(\Z)}$.

Cette observation \'etant faite, Juschenko et Monod donnent explicitement une suite $(f_n)_{n\ge 1}$, \`a savoir, en identifiant $\mathsf{C}_2$ et $\{0,1\}$
\begin{eqnarray*}
f_n\colon\qquad \{0,1\}^\Z & \to & ]0,1]\\
x=(x_j)_{j\in\Z} & \mapsto & \exp\left(-n\sum_{j\in\Z}x_je^{-|j|/n}\right)
\end{eqnarray*}

La condition (\ref{fnze}) se v\'erifie facilement, tandis que (\ref{fnai}) demande plus de travail (d'analyse \'el\'ementaire).

On va utiliser ici une autre suite de fonctions, qui nous a \'et\'e indiqu\'ee par M.~de la Salle, pour laquelle ces v\'erifications sont plus directes. On d\'efinit, pour $n\ge 1$
\begin{eqnarray*}
h_n\colon\qquad \{0,1\}^\Z & \to & [0,1]\\
x=(x_j)_{j\in\Z} & \mapsto & \prod_{j\in\Z}h_{nj}(x_j),
\end{eqnarray*}
avec
\[(h_{nj}(0),h_{nj}(1))= \sqrt{2}(\cos \theta_{nj},\sin \theta_{nj});      \qquad \theta_{nj} = \frac{\pi}4 \min\left(\sqrt{\frac{|j|}{n}},1\right)\in [0,\pi/4].  \]  

Remarquons que le produit d\'efinissant $h_n(x)$ est fini puisque $h_{nj} = 1$ si $|j|>n$. On observe \'egalement que $\theta_{n0}=0$, si bien que $h_{n0}(1)=0$ et $h_n(x)=0$ si $x_0=1$. En d'autres termes, on a $h_n=h_n\mathbf{1}_{H_0}$, ce qui donne (\ref{fnze}) (avec $h_n$ en lieu et place de $f_n$).
Notons que $\|h_n\|_2=1$ pour tout $n\ge 1$, puisque $\|h_{nj}\|_2=1$ pour tous $n$ et $j$ (l'ensemble $\{0,1\}$ \'etant muni de la mesure de Haar de $\mathsf{C}_2$). 
Il reste \`a prouver (\ref{fnai}), ou, ce qui revient au m\^eme:

\begin{propp}\label{mainprop2}
Pour tout $g\in\TM(\Z)$ on a
\[\lim_{n\to\infty}\langle h_n,g^{-1}h_n\rangle =1.\]
\end{propp}
\begin{proof}On a $g^{-1}h_n(x)=h_n(gx)$ et $(gx)_j=x_{g^{-j}(j)}$, si bien que $(g^{-1}h_n)(x)=\prod_{j\in\Z}h_{n\,g(j)}(x_j)$. 
On a donc, par le th\'eor\`eme de Fubini, la formule 
\[\langle h_n,g^{-1}h_n\rangle=\prod_{j\in\Z}\langle h_{nj},h_{n\,g(j)}\rangle=\prod_{j\in\Z}\cos(\theta_{nj}-\theta_{n\,g(j)}).\]
En utilisant l'in\'egalit\'e $\cos(x)\ge \exp(-x^2)$ valable pour tout $x\in [-\pi/4,\pi/4]$, on en d\'eduit
\begin{equation}\langle h_n,g^{-1}h_n\rangle\ge \exp\left(-\sum_{j\in\Z}(\theta_{nj}-\theta_{n\,g(j)})^2\right).\label{hn}\end{equation}
On va majorer la somme dans le terme de droite. En utilisant que
\[|\sqrt{j}-\sqrt{k}|\le \frac{|j-k|}{\max(1,\sqrt{j})},\quad\forall j,k\in\N,\]
on obtient, en posant $c=\sup_{j\in\Z}\big||g(j)|-|j|\big|$
\[|\theta_{nj}-\theta_{n\,g(j)}|\le \frac{\pi c}{4\sqrt{n}}\frac{1}{\max(1,\sqrt{|j|})}.\]
 Il en d\'ecoule, en remarquant en outre que $\theta_{nj}-\theta_{n\,g(j)}=0$ si $|j|\ge n+c$

\begin{align*}
\sum_{j\in\Z}(\theta_{nj}-\theta_{n\,g(j)})^2 = &
\sum_{|j|\le n+c}(\theta_{nj}-\theta_{n\,g(j)})^2\\
\le & \frac{\pi^2c^2}{16n}\sum_{|j|\le n+c}\frac{1}{\max(1,|j|)}\\
\le & \frac{\pi^2c^2}{16n}(3+\log(n+c))\\
& =O\left(c^2\frac{\log n}n+c^2\log c\right)=O_c\left(\frac{\log n}n\right)
\end{align*}
donc par (\ref{hn}) on obtient que $\langle h_n,g^{-1}h_n\rangle$ tend bien vers~1 \`a $g$ fix\'e.
\end{proof}

\begin{rema}
La preuve de la proposition \ref{mainprop2} montre que pour tout $g\in\TM(\Z)$ fix\'e, on a $1-\langle h_n,g h_n\rangle=O(\log(n)/n)$, ou de mani\`ere \'equivalente $\|h_n-gh_n\|_2=O((\log n/n)^{1/2})$. 

D'autre part, remarquons que la preuve de la proposition s'applique sans changement \`a tout $g\in\TM(|\Z|)$, o\`u ce dernier d\'esigne le groupe des permutations $g$ de $\Z$ telles que $\sup_{n\in\Z}\Big||g(n)|-|n|\Big|<\infty$. Ce groupe contient strictement $\TM(\Z)$, par exemple il contient l'involution $n\mapsto -n$ qui n'est pas dans $\TM(\Z)$; il est \'egalement isomorphe au sous-groupe $\TM(\N)$ de $\TM(\Z)$, comme on voit en conjuguant par la bijection 
\mbox{2-lipschitzienne} de $\Z$ sur $\N$ qui envoie $n\ge 0$ sur $2n$ et $-n\le -1$ sur $2n-1$.
\end{rema}

\appendix
\numberwithin{thm}{section} 
\section{Classes \'el\'ementaires}

Soit $\mathcal{C}$ une classe de groupes (invariante par isomorphie). On d\'efinit $\mathscr{E}(\mathcal{C})$ comme la plus petite classe de groupes contenant $\mathcal{C}$ et stable par isomorphie et passage aux sous-groupes, quotients, extensions et limites inductives filtrantes. On l'appelle classe \'el\'ementaire engendr\'ee par $\mathcal{C}$.

Le lemme suivant est bien connu.
\begin{lem}\label{lre}
On suppose $\mathcal{C}$ stable par passage aux sous-groupes et aux quotients. On d\'efinit par r\'ecurrence transfinie
\begin{itemize}
\item $\mathscr{E}_0(\mathcal{C})=\mathcal{C}$;
\item si $\alpha$ est un ordinal, $\mathscr{E}_{\alpha+1}(\mathcal{C})$ comme l'ensemble des groupes isomorphes \`a une extension \`a noyau dans $\mathscr{E}_\alpha(\mathcal{C})$ et quotient dans $\mathcal{C}$, ou \`a une limite directe (c'est-\`a-dire limite inductive filtrante d'injections) de groupes de $\mathscr{E}_\alpha(\mathcal{C})$;
\item $\mathscr{E}_\alpha(\mathcal{C})=\bigcup_{\beta<\alpha}\mathscr{E}_\beta(\mathcal{C})$ pour $\alpha$ ordinal limite.
\end{itemize}
Alors $\mathscr{E}(\mathcal{C})=\bigcup_\alpha\mathscr{E}_\alpha(\mathcal{C})$. En particulier $\mathscr{E}(\mathcal{C})$ est la plus petite classe de groupes stable par isomorphie et par passage aux limites directes et extensions \`a quotient dans $\mathcal{C}$.
\end{lem}
\begin{proof}
Par une r\'ecurrence transfinie imm\'ediate, chaque $\mathscr{E}_\alpha(\mathcal{C})$ est stable par passage aux sous-groupes, et toute limite directe de groupes dans $\mathscr{E}_\alpha(\mathcal{C})$ est dans $\mathscr{E}_{\alpha+1}(\mathcal{C})$, si bien que $\mathcal{C}'=\bigcup_\alpha\mathscr{E}_\alpha(\mathcal{C})$ est stable par passage aux sous-groupes, quotients et limites directes; la stabilit\'e \`a la fois par passage au quotient et aux limites directes implique la stabilit\'e par passage aux limites inductives filtrantes. Il reste \`a v\'erifier la stabilit\'e de $\mathcal{C}'$ par passage aux extensions. V\'erifions, par r\'ecurrence transfinie sur $\beta$, qu'il est stable par extension avec quotient dans $\mathscr{E}_\beta(\mathcal{C})$. C'est vrai pour $\beta=0$; on suppose que c'est acquis pour tout $\beta'<\beta$. Si $\beta$ est limite il n'y a rien \`a faire, on suppose donc $\beta=\gamma+1$. On consid\`ere donc un groupe $G$, extension avec noyau $N$ dans $\mathcal{C}'$ et quotient $Q$ dans $\mathscr{E}_{\gamma+1}(\mathcal{C})$. Si $Q$ est limite directe de groupes dans $\mathscr{E}_\gamma(\mathcal{C})$, alors $G$ est limite directe de leurs images r\'eciproques dans $G$, qui sont par r\'ecurrence dans $\mathcal{C'}$, et donc $G$ est dans $\mathcal{C}'$. Sinon, $Q$ est extension avec noyau $M$ dans $C_\gamma$ et quotient dans~$\mathcal{C}$. L'image r\'eciproque $M_1$ de $M$ dans $G$ est donc, par r\'ecurrence, dans $\mathcal{C}'$, et donc $G$, extension de noyau $M_1$ et noyau dans $\mathcal{C}$, l'est donc \'egalement.
\end{proof}

Le lemme peut \^etre utile pour montrer qu'un groupe n'appartient pas \`a la classe \'el\'ementaire engendr\'ee par $\mathcal{C}$. Le cas le plus \'evident (servant souvent sous sa forme contrapos\'ee) est le suivant:

\begin{fact}\label{factsim}
Supposons la classe $\mathcal{C}$ stable par passage aux sous-groupes et aux quotients. Si $G$ est un groupe simple de type fini dans $\mathscr{E}(\mathcal{C})$, alors $G$ est dans $\mathcal{C}$.
\end{fact}
\begin{proof}
Par hypoth\`ese, $G$ n'a aucune d\'ecomposition non triviale en extension ou en limite directe. Cela impose que l'ordinal minimal $\alpha$ tel que $G\in\mathscr{E}_\alpha(\mathcal{C})$ est~0.
\end{proof}

\begin{exe}
Soit $\mathcal{A}$ la classe des groupes de type fini ab\'eliens et $\mathcal{F}$ la classe des groupes finis. La classe $\mathscr{E}(\mathcal{A}\cup\mathcal{F})$, introduite par Day \cite[p.~520]{Da57} et notamment \'etudi\'ee par Chou \cite{Ch80} est appel\'ee classe des groupes {\em \'el\'ementairement moyennables}. Soit $\mathcal{SE}$ la classe des groupes de type fini \`a croissance sous-expo\-nentielle. La classe $\mathscr{E}(\mathcal{SE})$ est appel\'ee classe des groupes {\em sous-exponentiellement moyennables}.
On a \mbox{$\mathcal{A}\cup\mathcal{F}\subset \mathcal{SE}$,} si bien que $\mathscr{E}(\mathcal{A}\cup\mathcal{F})\subset \mathscr{E}(\mathcal{SE})$.

La classe des groupes moyennables est stable par passage aux extensions et aux limites directes, et contient les groupes \`a croissance sous-exponentielle, si bien que ces classes sont constitu\'ees de groupes moyennables (ce qui est l'unique justification de la terminologie!).
\end{exe}

Une autre application du lemme \ref{lre} est la suivante.

\begin{fact}
Soit $G$ un groupe de type fini \'el\'ementairement moyennable et infini. Alors $G$ a virtuellement un morphisme surjectif sur $\Z$, au sens o\`u il a un sous-groupe d'indice fini ayant un morphisme surjectif vers $\Z$. De mani\`ere \'equivalente, il poss\`ede un quotient infini et virtuellement ab\'elien.
\end{fact}
\begin{proof}
Montrons plus g\'en\'eralement que si $\mathcal{C}$ est une classe de groupes stable par passage aux sous-groupes et quotients telle que tout groupe infini dans $\mathcal{C}$ a virtuellement un morphisme surjectif sur $\Z$, alors il en est de m\^eme pour tout $G$ infini de type fini dans  $\mathscr{E}(\mathcal{C})$.

Supposons le contraire. Soit $\alpha$ minimal tel que $G\in\mathscr{E}_\alpha(\mathcal{C})$. Alors $\alpha$ n'est pas un ordinal limite et par hypoth\`ese, $\alpha\neq 0$, donc $\alpha=\beta+1$ est un successeur. Comme $G$ est de type fini, il est donc extension \`a noyau $N$ dans $\mathscr{E}_\beta(\mathcal{C})$ et quotient $Q$ dans $\mathcal{C}$. Si $Q$ est infini, alors il a virtuellement un morphisme surjectif sur $\Z$, et donc $G$ aussi. Si $Q$ est fini, alors $N$ est d'indice fini, et par minimalit\'e de $\alpha$ a virtuellement un morphisme surjectif sur $\Z$, et donc $G$ aussi.

Pour la reformulation, il est clair que l'existence d'un quotient virtuellement ab\'elien implique l'existence d'un sous-groupe d'indice fini avec un ab\'elianis\'e infini; la r\'eciproque est vraie: si $G$ admet un sous-groupe d'indice fini $H$ d'ab\'elianis\'e infini, alors si $N$ est l'intersection des conjugu\'es de $[H,H]$ dans $G$ alors $G/N$ est virtuellement ab\'elien et infini.
\end{proof}

\section{Ajout (f\'evrier 2020)}\label{ajout} 

\vspace{-.5cm}
{(\it En dehors des ajouts dat\'es de f\'evrier 2020 li\'es \`a cet 
appendice, la version co\"incide avec celle envoy\'ee \`a 
l'\'editeur en mars 2013.)}

\smallskip

St\v ep\' anek et Rubin (1989) \cite{SR} ont introduit de faÃ§on 
g\'en\'erale le groupe plein-topologique d'un sous-groupe du groupe des 
autohomÃ©omorphismes d'un espace compact totalement s\'epar\'e. Ils se placent dans le cadre 
du groupe des automorphismes d'une alg\`ebre bool\'eenne, mais utilisent 
explicitement la dualit\'e de Stone. Ils d\'emontrent \cite[Lemma 
5.17]{SR}: {\it Soit $X$ un espace compact totalement s\'epar\'e. Soit $\Gamma$ un 
sous-groupe libre non trivial de $\mathrm{Homeo}(X)$ agissant 
topologiquement librement sur $X$. Alors $[\![\Gamma]\!]$ (not\'e 
$\mathrm{PW}(\Gamma)$ ``elements piecewise in $\Gamma$'') n'est pas 
simple.} Plus pr\'ecis\'ement, ils construisent un morphisme non nul 
$[\![\Gamma]\!]\to\Z$. En fait la preuve fonctionne en supposant 
seulement qu'il existe $x_0\in X$ de stabilisateur trivial, et qu'il 
existe une partie $M$ de $\Gamma$ telle que pour tout $g\in\Gamma$ on 
ait $gM\bigtriangleup M$ fini et que de plus le morphisme $\Gamma\to\Z$, 
$g\mapsto |gM\smallsetminus M|-|M\smallsetminus gM|$ est non nul. En 
effet, ce morphisme se prolonge en un morphisme $[\![\Gamma]\!]\to\Z$ 
donn\'e par $f\mapsto |fM'\smallsetminus M'|-|M'\smallsetminus fM'|$ 
avec $M'=Mx_0$. La construction au d\'ebut du \S\ref{abelianise} est en fait un cas 
particulier de cette situation.


\begin{thebibliography}{666}



\bibitem[Ab06]{Abert} M. Abert. Representing graphs by the non-commuting relation, Publ. Math. Debrecen 69/3 (2006), 261--269.

\bibitem[BeM08]{BeM} S. Bezuglyi, K. Medynets. Full groups, flip conjugacy, and orbit equivalence of Cantor minimal systems. Colloq. Math. 110(2) (2008), 409--429.

\bibitem[BKN10]{BKN} L. Bartholdi, V. Kaimanovich, V. Nekrashevych.
On amenability of automata groups.
Duke Math. J. 154(3) (2010), 575--598.

\bibitem[Bo83]{Bo} M. Boyle. Topological orbit equivalence and factor maps in 
symbolic dynamics. Ph.D. thesis, Univ. of Washington,  1983.

\bibitem[BoTo98]{BoTo} M. Boyle, J. Tomiyama. Bounded topological orbit equivalence and $C^*$-algebras, J. Math. Soc. Japan 50(2) (1998), 317--329.

\bibitem[Bou]{Bou} N. Bourbaki. Topologie g\'en\'erale, chapitres 1 \`a 4. Hermann, Paris, 1971.\break
R\'eimpression Springer, 2007.

\bibitem[Bri09]{Bri} J. Brieussel.
Amenability and non-uniform growth of some directed automorphism groups of a rooted tree.
Math. Z. 263(2) (2009), 265--293.

\bibitem[Bro87]{Bro} K. Brown. Finiteness properties of groups. In: Proceedings of the Northwestern Conference on Cohomology of Groups, Evanston, IL, 1985, J. Pure Appl. Algebra 44 (1987), 45--75.

\bibitem[BV05]{BV} L. Bartholdi, B. Vir\'ag. Amenability via random walks. Duke Math. J. 130(1) (2005), 39--56.

\bibitem[Ch80]{Ch80} C. Chou. Elementary amenable groups. Illinois J. Math. 24(3) (1980), 396--407.

\bibitem[Da57]{Da57} M. M. Day. Amenable semigroups.
Illinois J. Math. Volume 1, Issue 4 (1957), 509--544.

\bibitem[DFG11]{DFG} F. Dahmani, K. Fujiwara, V. Guirardel. Free groups of interval exchange transformations are rare. ArXiv math/1101.5909v2 (2011), \`a para\^\i tre \`a Groups Geom. Dyn.

\bibitem[Do90]{VD} E. van Douwen, Measures invariant under actions of $F_2$. Topology Appl. 34(1) (1990), 53--68.

\bibitem[EM12]{EM} G. Elek, N. Monod. On the topological full group of minimal Cantor $\Z^2$-systems. ArXiv math/1201.0257v2 (2012); \`a para\^\i tre \`a Proc. Amer. Math. Soc.

\bibitem[Ers06]{Ers} A. Erschler.
Piecewise automatic groups.
Duke Math. J. 134(3) (2006), 591--613. 

\bibitem[Eym]{Eym} P. Eymard. Initiation \`a la th\'eorie des groupes moyennables. Pages 89--107 dans: Analyse harmonique sur les groupes de Lie. Lecture Notes in Math. 497, Springer (1975).

\bibitem[GlW95]{GW} E. Glasner, B. Weiss. Weak orbit equivalence of Cantor minimal systems, Internat. J. Math. 6(4) (1995), 559--579.

\bibitem[GPS95]{GPS95} T. Giordano, I. F. Putnam, C. F. Skau. Topological orbit equivalence and $C^*$-crossed products, J. Reine Angew. Math. 469 (1995), 51--111.

\bibitem[GPS99]{GPS} T. Giordano, I. F. Putnam, C. F. Skau. Full groups of Cantor minimal systems.
Israel J. Math. 111 (1999), 285--320.

\bibitem[Gre]{Gre} F. Greenleaf. Invariant means on topological groups and their applications. New York, Van Nostrand, 1969.

\bibitem[Gri84]{Gri} R. Grigorchuk. Degrees of growth of finitely generated groups, and the theory of invariant means. Izvestiya: Mathematics, 25(2):259--300, 1985.
Russian original: Izv. Akad. Nauk SSSR Ser. Mat. 48(5) (1984), 939--985.

\bibitem[GrMe12]{GrMe} R. Grigorchuk, K. Medynets. On algebraic properties of topological full groups. ArXiv math/1105.0719v4 (2012).

\bibitem[GrZu02]{GZ} R. Grigorchuk, A. Zuk. On a torsion-free weakly branch group defined by a three state automaton. Int. J. Algebra Comput. 12  (2002), 223--246.

\bibitem[GTY10]{GTY} E. Guentner, R. Tessera, G. Yu. A notion of geometric complexity and its applications to topological rigidity. ArXiv math/1008.0884v1 (2010), \`a para\^\i tre \`a Inv. Math.

\bibitem[HS86]{HS86} P. de la Harpe, G. Skandalis. Un r\'esultat de Tarski sur les actions moyennables de groupes et les partitions paradoxales. Enseign. Math., II. S\'er. 32 (1986), 121--138.

\bibitem[Hou78]{Hou} C. Houghton, The first cohomology of a group with permutation module coefficients, Arch. Math. (Basel) 31(3) (1978), 254--258.

\bibitem[JM12]{JM} K. Juschenko, N. Monod. Cantor systems, piecewise translations and simple amenable groups. ArXiv math/1204.2132v3, \`a para\^\i tre \`a Ann. of Math.

\bibitem[Kai05]{Kai} V. Kaimanovich.
``M\"unchhausen trick'' and amenability of self-similar groups.
Internat. J. Algebra Comput. 15 (2005), 907--937. 

\bibitem[Kea75]{Kea} M. Keane. Interval exchange transformations. Math. Z. 141(1) (1975), 25--31.

\bibitem[KH]{KH} A. Katok, B. Hasselblatt. Introduction to the modern theory of dynamical systems.
Encyclopedia of Mathematics and its Applications 54, Cambridge University Press, 1995.

\bibitem[Kri80]{Kri} W. Krieger. On a dimension for a class of homeomorphism groups. Math. Ann. 252 (1980), 87--95.

\bibitem[Luk]{Luk} J. Lukes. Integral representation theory: applications to convexity, Banach spaces and potential theory. Walter de Gruyter, 2010.

\bibitem[Ma06]{Ma1} H. Matui. Some remarks on topological full groups of Cantor minimal systems. Internat. J. Math. 17(2) 231--251, 2006.

\bibitem[Ma12]{Ma2} H. Matui. Some remarks on topological full groups of Cantor minimal systems II. ArXiv math/1111.3134v1, \`a para\^\i tre \`a Ergodic Th. Dynam. Systems.

\bibitem[MH38]{MH} M. Morse, G. Hedlund. Symbolic Dynamics. Amer.\ J.\ Math.\ 60(4) (1938), 815--866.

\bibitem[Neu37]{Neu} B. H. Neumann. Some remarks on infinite groups. Proc. London Math. Soc. (2), 12 (1937), 120--127.

\bibitem[Nov09]{Nov} C. F. Novak. Discontinuity-growth of interval-exchange maps. J. Mod. Dyn. 3(3) (2009), 379--405.

\bibitem[Pat]{Pat} A. Paterson. Amenability. Math. Surveys Monogr. Amer. Math. Soc., 1988.

\bibitem[Pu89]{Pu} I. F. Putnam. The $C^*$-algebras associated with minimal homeomorphisms of the Cantor set. Pacific J. Math. Volume 136, Number 2 (1989), 329--353.

\bibitem[Ros74]{Ros} J. Rosenblatt. Invariant measures and growth conditions. Trans. Amer. Math. Soc. 193 (1974), 33--53 .

\bibitem[Sab74]{Sab} G. Sabbagh. Caract\'erisation alg\'ebrique des groupes de type fini ayant un probl\`eme du mot r\'esoluble (th\'eor\`eme de Boone-Higman, travaux de B.~H. Neumann et MacIntyre). S\'eminaire N. Bourbaki 1974-1975, exp. n\textordmasculine 457, Lecture Notes in Math.
514 (1976), 61--80.

\bibitem[Sag97]{Sag} M. Sageev. Codimension-1 subgroups and splittings of groups, J. of Algebra 189 (1997), 377--389.

\bibitem[Sch27]{Sch} O. Schreier. Die Untergruppen der freien Gruppen. Abhandlungen Math. Hamburg 5 (1927), 161--183.

\bibitem[Sk97]{Sk} C. F. Skau. Orbit structure of topological dynamical systems and its invariants, Operator Algebras and Quantum Field Theory, Rome (1996) (International Press, Cambridge, MA, 1997),  533--544.

\bibitem[SR]{SR} P. St\v ep\' anek, M.  Rubin. Homogeneous Boolean algebras. Handbook of Boolean algebras, Vol. 2, 679--715, North-Holland, Amsterdam, 1989. {\it (RÃ©fÃ©rence ajoutÃ©e, fÃ©vrier 2020)}

\bibitem[SS]{SS} L. Steen, J. Seebach. Counterexamples in topology. Dover Publications, New York, 1978.


\bibitem[St84]{St2} A. St\"epin.
A remark on the approximability of groups. (En russe) 
Vestnik Moskov. Univ. Ser. I Mat. Mekh. 1984, 85--87. 
   Traduction anglaise : Moscow Univ. Math. Bull. 39 (1984), 90--93.

\bibitem[Ve81]{Ve1} A. Vershik. Uniform algebraic approximations of shift and multiplication operators.
Dokl. Akad. Nauk SSSR 259(3), 526--529 (1981).  Traduction anglaise :
Sov. Math. Dokl. 24 (1981), 97--100.

\bibitem[Ve82]{Ve2} A. Vershik. A theorem on periodical Markov approximation in ergodic theory.
Zapiski Nauchn. Semin. LOMI 115 (1982), 72--82 . Traduction anglaise : J.
Sov. Math. 28 (1985), 667--674. Autre traduction anglaise dans ``Ergodic
theory and related topics'' (Vitte, 1981). Math. Res. 12 (1982), 195--206.
Akademie-Verlag, Berlin.
\end{thebibliography}
\end{document}